\documentclass[11pt]{amsart}

\usepackage{amsmath,amssymb}
\usepackage{mathrsfs}

\usepackage[english]{babel}
\usepackage[utf8x]{inputenc}
\usepackage[T1]{fontenc}
\usepackage[new]{old-arrows}

\usepackage[top=3cm,bottom=2cm,left=3cm,right=3cm,marginparwidth=1.75cm]{geometry}
\usepackage{graphicx}
\usepackage[colorinlistoftodos]{todonotes}
\usepackage[colorlinks=true, allcolors=blue]{hyperref}
\usepackage{cite}
\usepackage{enumitem}
\setlist[1]{itemsep=3pt}
\usepackage[all,cmtip]{xy}
\usepackage{tikz-cd}
\usepackage{cleveref}
\usepackage{autonum}

\newtheorem{theorem}{Theorem}
\numberwithin{theorem}{section}
\newtheorem{proposition}[theorem]{Proposition}
\newtheorem{corollary}[theorem]{Corollary}
\newtheorem{lemma}[theorem]{Lemma}

\theoremstyle{definition}
\newtheorem{definition}[theorem]{Definition}

\theoremstyle{remark}
\newtheorem{remark}[theorem]{Remark}
\newtheorem{example}[theorem]{Example}
\crefname{equation}{}{}
\crefname{equation}{}{}
\crefname{figure}{Figure}{Figures}
\crefname{section}{Section}{Sections}
\crefname{lemma}{Lemma}{Lemmata}
\crefname{proposition}{Proposition}{Propositions}
\crefname{theorem}{Theorem}{Theorems}
\crefname{corollary}{Corollary}{Corollaries}
\crefname{definition}{Definition}{Definitions}
\crefname{remark}{Remark}{Remarks}
\crefname{proposition}{Proposition}{Proposition}
\crefname{corollary}{Corollary}{Corollaries}
\crefname{example}{Example}{Examples}
\crefname{conjecture}{Conjecture}{Conjectures}

\newcommand{\R}{\mathbb{R}}
\newcommand{\C}{\mathbb{C}}
\newcommand{\N}{\mathbb{N}}
\newcommand{\ZZ}{\mathcal{Z}}
\newcommand{\ZZo}{\mathcal{Z}}
\newcommand{\KK}{\mathcal{K}}
\newcommand{\cP}{\mathcal{P}}

\newcommand{\dd}{\mathrm{d}}

\newcommand{\PP}{\mathbb{P}}
\newcommand{\Z}{\mathcal{Z}}

\newcommand{\K}{\mathcal{K}}

\newcommand{\EE}{\mathbb{E}}
\newcommand{\A}{\mathcal{A}}
\newcommand{\MV}{\mathrm{MV}}
\newcommand{\y}{Y}
\newcommand{\VZ}{\widehat{\mathcal{Z}}}
\newcommand{\VK}{\widehat{\mathcal{K}}}
\newcommand{\eps}{\varepsilon}
\newcommand{\st}{\ |\ }

\newcommand{\GZ}{\mathcal{G}}
\newcommand{\VGZ}{\widehat{\mathcal{G}}}
\newcommand{\GA}{\mathcal{G}}

\newcommand{\cM}{\mathcal{M}}
\newcommand{\cH}{\HH}

\newcommand{\mT}{\widetilde{T}}

\newcommand{\sg}[1]{\tfrac{1}{2}[- #1 , #1]}

\newcommand{\ExpZon}[1]{K({#1})}
\newcommand{\timeszon}{\wedge}
\newcommand{\wc}{\wedge_\C}

\DeclareMathOperator{\vol}{vol}
\DeclareMathOperator{\Span}{span}
\DeclareMathOperator*{\mean}{\mathbb{E}}

\DeclareMathOperator{\HH}{H}

\newcommand{\spanK}[1]{\langle #1\rangle}

\usepackage{youngtab}

\numberwithin{equation}{section}
\newcommand{\be}{\begin{equation}}
\newcommand{\ee}{\end{equation}}

\title[The zonoid algebra]{The zonoid algebra, generalized mixed volumes,\\
and random determinants}

\author{Paul Breiding}\thanks{Br: Max-Planck-Institute for Mathematics in the Sciences Leipzig, paul.breiding@mis.mpg.de. Funded by the Deutsche Forschungsgemeinschaft (DFG, German Research Foundation) -- Projektnummer 445466444}
\author{Peter B\"urgisser}\thanks{B\"u: Supported by the ERC under the European Union's Horizon 2020 research and innovation programme (grant agreement no. 787840).}
\author{Antonio Lerario}
\author{L\'eo Mathis}

\begin{document}

\begin{abstract}
We show that every multilinear map between Euclidean spaces induces a unique, continuous,
Minkowski multilinear map of the corresponding real cones of zonoids.
Applied to the wedge product of the exterior algebra of a Euclidean space,
this yields a multiplication of zonoids, defining the structure of a commutative, associative, and partially ordered ring,
which we call the {\em zonoid algebra}.
This framework gives a new perspective on classical objects in convex geometry, and it allows to introduce new functionals on zonoids, in particular generalizing the notion of mixed volume. We also analyze a similar construction based on the complex wedge product,
which leads to the new notion of {\em mixed~$J$--volume}.
These ideas connect to the theory of random determinants.
\end{abstract}

\maketitle

\section{Introduction}
This paper is at the interface of convex geometry and probability. It links the three topics zonoids,
mixed volumes and random determinants under a general new framework.

In convex geometry, sums of line segments are called \emph{zonotopes} and their limits
with respect to the Hausdorff metric are called \emph{zonoids}.
Zonoids form a special, yet fundamental class of convex bodies~\cite{bible}.
In functional analysis, probability, combinatorics, and algebraic geometry,
it is sometimes possible to reformulate questions in terms of convex bodies,
typically involving the notion of mixed volume.
In this paper we show that there is a richer multiplicative structure as soon as we confine
ourselves to zonoids.

Let $V$ be a Euclidean space of dimension~$m$ and~$\Z(V)$ denote the set of its
\emph{centrally symmetric zonoids}, which in the following we simply refer to as zonoids.
On~$\Z(V)$ we have the Minkowski addition and the multiplication with nonnegative reals,
and the \emph{Hausdorff distance} makes it a complete metric space.
We study more closely the \emph{tensor products} of zonoids.
More specifically,
if $V_1,\ldots,V_p$ are Euclidean spaces, we define a Minkowski multilinear map
\be\label{eq:def-ML-map}
 \Z(V_1)\times \cdots \times \Z(V_p)\to \Z(V_1\otimes \cdots \otimes V_p),
\ee
which is monotone with respect to inclusion and continuous (this is the same tensor product that was defined earlier by Aubrun and Lancien \cite{AubrunLancien}).
Based on the tensor product, we prove in \cref{thm:FTZC} that
each multilinear map between Euclidean spaces induces
a multilinear map of the corresponding spaces of zonoids.
The resulting  maps are again continuous and monotone with respect to inclusion.

A main contribution of this paper is to define and study the \emph{zonoid algebra}.
For this, we first pass from the cone of zonoids $\Z(V)$ to a vector space $\VZ(V)$,
which we call \emph{virtual zonoids}.
In simple terms, $\VZ(V)$ consists of formal differences of elements in $\Z(V)$, but this abstract vector space can be naturally embedded  into various classical vector spaces, e.g. in the space of measures, or the space of functions on the sphere, see \cref{sec:virtualzonoids,sec:VZandmeas}. In \cref{sec:FTZC} we introduce the zonoid algebra over $V$. As a vector space this is the direct sum of the vector spaces of virtual zonoids in all exterior powers of $V$:
$$
 \A(V) :=  \bigoplus_{k=0}^m \widehat{\Z}(\Lambda^k V).
$$
The multiplication in $\A(V)$ is the multilinear map on zonoids induced by the \emph{wedge products}
$\wedge: \Lambda^k(V)\times \Lambda^l(V)\to \Lambda^{k+l}(V), \, (v,w) \mapsto v\wedge w$
for $k+l\leq m$. Thus, if $K\in \Z(\Lambda^k(V))$ and $L\in \Z(\Lambda^l(V))$ are zonoids, we can \emph{multiply} them and get the zonoid $K\wedge L\in \Z(\Lambda^{k+l}V)$. In particular, the product $K_1\wedge \cdots\wedge K_p$ of $d$ zonoids $K_1,\ldots,K_d\in\Z(V)$ is a zonoid in $\Lambda^{d}V$.

The zonoid algebra $\A(V)$ has a rich structure: we prove in \cref{th:zonoidalgebra} that it is a graded,
commutative partially
ordered algebra.
While the Minkowski multilinear map~\eqref{eq:def-ML-map} is continuous, its extension to virtual zonoids requires careful topological considerations,
which we discuss in \cref{sec:continuity}.
We also introduce a subalgebra $\GA(V)$ (\cref{def:GA}), which we call \emph{Grassmann zonoid algebra},
and which plays a crucial role in the probabilistic intersection theory developed in our
follow-up work~\cite{BBML2021}.

The connection to the mixed volume enters via the concept of the \emph{length}
$\ell\colon \A(V) \to \R$, which is
a monotone, continuous linear functional; see \cref{mydef:length} and \cref{th:lengthLF}. In fact, the length of a zonoid is its first intrinsic volume.
More generally, we show in \cref{th:IVL} that if $K\in\Z(V)$ is a zonoid, then $\frac{1}{d!}\ell(K^{\wedge d})$
equals its $d$-th intrinsic volume.
Moreover, we show in \cref{thm:FTZCdet} that, if $K_1,\ldots,K_m$ are zonoids in~$V$, then their mixed volume satisfies
\be\label{MV_is_length}
 \MV(K_1,\ldots,K_m) = \frac{1}{m!}\, \ell(K_1\wedge \cdots \wedge K_m).
\ee
For higher degree zonoids, i.e., zonoids in $\Lambda^k(V)$ with $1< k < m$,
the length of their product therefore naturally defines a generalized mixed volume.

For zonoids in $\C^n\cong \R^{2n}$, we can repeat the same construction using the \emph{complex wedge product},
seen as
a real multilinear map, $\wc: \Lambda^k(\C^n)\times \Lambda^l(\C^n)\to \Lambda^{k+l}(\C^n)$.
In \cref{sec:mixed_J_volume} we define what we call the \emph{mixed $J$--volume}
\be\label{MV_J_is_length}
\MV^J(K_1,\ldots,K_n) := \frac{1}{n!}\,\ell(K_1\wc \cdots \wc K_n),
\ee
where $K_1,\ldots,K_n$ are zonoids in $\C^n\cong \R^{2n}$.
When the $K_i$ are real, i.e., contained in $\R^n\subseteq\C^n$, then
we retrieve the classical mixed volume $\MV$.
Furthermore, we prove in \cref{prop:volJisval} that the $J$--volume
$\vol^J_n(K):= \frac{1}{n!}\,\ell(K^{\wc n})$ can be extended to polytopes and
that it is a weakly continuous, translation invariant valuation.
By contrast, using a recent result by Wannerer \cite{Wannerer}, we prove in \cref{cor:noextofJ} that $\vol^J_n$
cannot be extended to a continuous valuation on all convex bodies in $\C^n$. We also relate our definition to \emph{Kazarnovskii's pseudovolume}~\cite{Ka1}; see \cref{def:Kaza}.

We apply our theory to the study of the expected absolute value
$|\det [X_1,\ldots,X_n]|$
of the determinant of random matrices, whose
column vectors $X_{1},\ldots,X_{m}\in\R^m$ are random vectors.
Vitale \cite{Vitale} gave a formula for this expectation in the case
where the $X_i$ are i.i.d.\  random vectors.
His formula is in terms of the volume of a zonoid: to each random vector~$X$
in~$\R^m$
which is {\em integrable}, i.e., $\mean \Vert X\Vert < \infty$,
there corresponds an associated zonoid $K(X)$; see \cref{def:K}.
Although~$X$ is a random vector, $K(X)$ is  deterministic. 
Vitale's result asserts that
$
\mean \big\vert \det
(M)\big\vert = m!\cdot \vol_m(K(X))$,
where
$
M=[\begin{smallmatrix} X_1& \ldots & X_m\end{smallmatrix}],
$
with i.i.d.\ random variables $X_i$ having the same distribution as $X$,
and where $\vol_m$ denotes the $m$-dimensional volume in $\R^m$.
It is important to emphasize that Vitale's result makes no assumption
on the independence of the entries of the random vector $X$.
But the columns must be independent, and identically distributed.
If the columns are independent but not identically distributed,
we can generalize Vitale's result as follows:
\be
\label{vitale_mv}\mean_{X_1,\ldots,X_m \atop \text{are independent}} \big\vert \det
(M)\big\vert = \ell(K(X_1)\wedge \cdots\wedge K(X_m)),\quad \text{ where }
M= \det [X_1,\ldots,X_m].
\ee
If the columns are not independent, the situation is far more complicated.
Explicit formulas are known for special cases; see, e.g., \cite{Girko1990},
but general formulas like \cref{vitale_mv} so far were not available.
We fill this gap by showing in \cref{thm:genvitale} that, if
$M=[M_{1}, \ldots, M_p]$ is a random $m\times m$ matrix partitioned into independent blocks
$M_j= [v_{j,1},\ldots,v_{j,d_j}]$
of size $m\times d_j$, and if~$Z_j=v_{j,1}\wedge \cdots \wedge v_{j, d_j}\in \Lambda^{d_j}\R^m$
are integrable, then
\be\label{gen_vitale_eq}
\EE |\det( M)|=\ell(K(Z_1)\wedge \cdots\wedge K(Z_p)).
\ee
This formula links the expected determinants of a matrix with independent blocks to
the multiplication in the zonoid algebra $\A(\R^m)$, which can be studied using methods
from convex geometry and commutative algebra \cite{Eisenbud}. Furthermore, using the complex wedge product, we obtain in \cref{thm:vitaleC} a new formula for the expected absolute value of the determinant of a random \emph{complex} matrix: $\EE |\det( M)|=\ell(K(Z_1)\wedge_{\mathbb C} \cdots\wedge_{\mathbb C} K(Z_p))$, where the $Z_j\in \Lambda_{\mathbb C}^{d_j}\C^n$ now are complex random vectors. Notice, that if the $d_j$ are all equal to one, we obtain the mixed
$J$--volume from \cref{MV_J_is_length},
and if the $K(Z_j)$ are all real zonoids, this specializes to \cref{gen_vitale_eq}.

We remark that~\cref{vitale_mv} has interesting consequences when combined with the \emph{Alexandrov--Fenchel inequality}.
For instance, if $X,Y\in \R^m$ are independent integrable random vectors,
then the expected area of the triangle $\Delta(X,Y)$,
whose vertices are the origin and $X$ and~$Y$, satisfies
$
\big(\mean \vol_2(\Delta(X,Y))\big)^2 \geq
\mean \vol_2(\Delta(X,X'))
\cdot \mean \vol_2(\Delta(Y,Y')),
$
where $X\sim X'$ and $Y\sim Y'$ are independent of $X,Y$, respectively.
We show this in \cref{AF_for_parallelotopes}.
Another remarkable consequence of our result is that we obtain a simple proof for zonoids
of the recent \emph{reverse Alexandrov--Fenchel inequality} by B\"or\"oczky and Hug~\cite{B-Hug:21}.
Namely, if $K_1,\ldots,K_p\in\Z(V)$ are zonoids and we have integers $d_1+\ldots+d_p=m=\dim(V)$,
then
$\MV(K_1[d_1],\ldots,K_p[d_p]) \leq \tfrac{1}{m!} \ell(K_1^{\wedge d_1})\cdots \ell(K_p^{\wedge d_p})$;
here, $K_i[d_i]$ indicates that $K_i$ appears $d_i$-many times.
We prove this in \cref{hug_result} and we also
characterize when equality holds.
We discuss more applications of the Alexandrov--Fenchel inequality to our theory in \cref{sec:MVandL}.

We conclude by pointing out that this work lays out the foundations for
the follow-up article~\cite{BBML2021}.
This will aim at developing a probabilistic intersection theory for compact homogeneous spaces,
in which the zonoid algebra takes the role played by the Chow ring or the cohomology ring in
the classical case. To each subvariety in the space we associate a zonoid in the exterior algebra of the tangent space at a distinguished point. The codimension of the variety equals the degree of the exterior power in which the zonoid lives. Then, the intersection of subvarieties in random position can be identified with the product of their zonoids.
This allows to develop further the
probabilistic Schubert calculus initiated in~\cite{PSC}.

\noindent\textbf{Outline.}
In \cref{sec:zonoids} we recall basic facts about zonoids, emphasizing how to represent them by random vectors following an approach by Vitale. We discuss the notion of the length of zonoids and discuss some of its properties. Moreover, we define the vector space of virtual zonoids and study its topology. In \cref{se:tensorproduct} we give the definition of the tensor product map for zonoids, and we discuss its continuity. In \cref{sec:FTZC} we introduce and discuss the zonoid algebra, and we prove that any multilinear map induces a unique multilinear map of virtual zonoids. In \cref{sec:MV_real} we explain how to relate our construction to intrinsic volumes, mixed volumes and random determinants. Finally, in \cref{sec:mixed_J_volume} we introduce and study the new notion of mixed $J$--volume.

\noindent\textbf{Acknowledgements.}
We would like to thank A.\ Bernig for pointing out the definition of Kazarnovskii's pseudovolume and its connection to our work. We also thank A.\ Khovanskii, G.\ Aubrun, C.\ Lancien and Y. Martinez-Maure for insightful discussions. Finally, we thank the anonymous referee for helpful comments.

\section{Zonoids}\label{sec:zonoids}

In this section we collect known facts about convex bodies,
for which we refer the reader to Schneider's extensive monograph~\cite{bible} for more details.
We also review and further extend a probabilistic description of zonoids
due to Vitale~\cite{Vitale}, which is a crucial tool in our paper.

Throughout the section, $V$ denotes
$m$-dimensional
Euclidean space
with inner product $\langle \;,\;\rangle$. The corresponding norm of $v\in V$
is given by $\Vert v\Vert = \langle v,v\rangle^\frac{1}{2}$.
The norm defines a topology on $V$ making it a topological vector space.
We denote by $B(V)\subset V$ the closed unit ball centered at zero in~$V$
and by $S(V)$ the unit sphere in $V$. We abbreviate $B^m:=B(\R^m)$ and
$S^{m-1}:=S(\R^m)$ for $V=\R^m$ with the standard inner product.
We denote by
\be [x,y]:=\{(1-t)x+ty \mid t\in [0,1]\}\ee
the \emph{segment} joining $x$ and $y\in V$.

\subsection{Convex bodies and support functions}

A subset $K\subset V$ is \emph{convex}, if for all $x,y\in K$ also $[x,y]\subset K$.
A nonempty compact convex set $K\subset V$ is called \emph{convex body}
and $\KK(V)$ denotes the set of convex bodies in $V$.
Given two convex bodies $K_1,K_2\in \KK(V)$, we can form
their \emph{Minkowski sum},
\be K_1 + K_2 :=\{x+y \mid x\in K_1, y\in K_2\},\ee
and, for $\lambda\ge 0$ the scalar multiple,
\be\label{scalar_multiplication}
\lambda K := \{\lambda  x\mid x\in K\}.
\ee
In particular, the Minkowski addition turns $\KK(V)$ into a commutative monoid.
Note that $\epsilon B(V)$ is the ball of radius $\epsilon>0$ centered at zero.
Moreover, we denote by
$\spanK{K}$ the linear span of~$K$.
The \emph{dimension} of $K$ is defined as the dimension of its affine span (since all convex bodies considered in this paper contain the origin, this will be the the same as the dimension of $\langle K \rangle$).
We set $\KK^m:=\KK(\R^m)$
for $V=\R^m$ with the standard inner product.

The \emph{Hausdorff distance} between $K_1, K_2\in \KK(V)$ is defined by
\begin{equation}\label{def_Hausdorff_metric}
d_H(K_1, K_2) := \inf\{\epsilon >  0\mid K_1 \subset
  K_2 + \epsilon B(V) \text{ and }  K_2 \subset K_1 + \epsilon B(V)\}.
\end{equation}
This leads to the \emph{norm} of a convex body:
\be\label{eq:normB} \|K\|:=d_{H}(K, \{0\})=\max_{x\in K}\|x\|.\ee
We also call $\|K\|$ the \emph{radius} of $K$.
 The Hausdorff distance \cref{def_Hausdorff_metric} makes $\KK(V)$ a complete metric space,
 see \cite[Theorem 1.8.5]{bible}.

A convex body $K\in \KK(V)$ can be described by its \emph{support function} $h_K:V\to \R$ defined by:
\be\label{def_h}
    h_K(v):=\max\{\langle v, x\rangle \mid x\in K\}.
\ee
This function satisfies
$h(v+w) \le h(v) + h(w)$ and $h(\lambda v) = \lambda h(v)$
for $v,w\in V$ and $\lambda\ge 0$.
Functions $h\colon V \to \R$ satisfying these properties are called
\emph{sublinear}.

All properties of convex bodies can be concisely described in terms of their support functions.
We summarize the relevant facts in the next well known proposition.
Let $C(S(V))$ denote the space of real valued, continuous functions
on the unit sphere $S(V)$ of $V$.
We think of $C(S(V))$ as a Banach space with respect
to either the $1$- or the $\infty$-norm. That is, for $f\in C(S(V))$:
\be\label{L_norms}
\Vert f\Vert_{1}= \frac{1}{\vol_{m-1}(S(V))}\, \int_{S(V)} \vert f(x) \vert \, \lambda(\mathrm{d}x) \quad \text{and}\quad \Vert f\Vert_{\infty}=\max_{x\in S(V)} \vert f(x)\vert,
\ee
where $\lambda$ is the usual Lebesgue measure on the sphere and $\vol_{k}$ denotes the $k$-dimensional volume with respect to $\lambda$.
We denote by $\bar{h}_K:=(h_K)|_{S(V)}$ the restriction of
the support function~$h_K$ to the unit sphere~$S(V)$.

\begin{proposition}\label{useful_properties_for_h}
The map
$h\colon (\K(V), d_H)
\to (C(S(V)), \Vert \cdot\Vert_{\infty}),\,
  K\mapsto \bar{h}_K$
is a Minkowski linear isometric embedding. That is,
for $K,L\in \KK(V)$ and $\lambda \geq 0$,
$$
h_{K + L} = h_K + h_L, \quad h_{\lambda K} = \lambda h_K , \quad
d_H(K,L) = \Vert h_K-h_L\Vert_{\infty}.
$$
In particular,
$\|K\| = \|\bar{h}_K\|_{\infty}$. Moreover,
\begin{enumerate}
\item
The image of $h$ consists of the sublinear functions on $V$ restricted to $S(V)$.
\item The support function $h_K$ allows to reconstruct the body $K$ as follows:
\be K = \{x\in V\mid \forall v \in V\: \langle x,v\rangle \leq h_K(v) \}.\ee
\item $K\subset L$ if and only if $h_K(u) \leq h_L(u)$ for all $u\in V$.
\item Let $W$ be a Euclidean space and $M\colon V\to W$ be linear with
adjoint $M^T\colon W\to V$.
Then $M(K)$ is a convex body with support function $h_{M(K)}(v)=h_K( M^T v)$.
\end{enumerate}
\end{proposition}

\begin{proof}
The three statements in display are \cite[Theorem 1.7.5]{bible}
and \cite[Lemma 1.8.14]{bible}.
The characterization (1)
of support functions as restrictions of sublinear functions
is~\cite[Theorem 1.7.1]{bible}.
The statement (2) expresses duality and is in the proof of \cite[Theorem 1.7.1]{bible}.
The third statement follows from the second item.
The last point is a direct consequence of the definition of the support function in \cref{def_h}.
\end{proof}

We will be mainly interested in centrally symmetric convex bodies~$K$,
characterized by $(-1)K =K$. They form a closed subset of $\KK(V)$.
The fourth item in \cref{useful_properties_for_h} implies that $K$ is centrally symmetric
if and only if $h_K(-v)=h_K(v)$ for all $v\in V$.

\subsection{Zonoids and random convex bodies}

A \emph{zonotope} $K$ in $V$ is defined as the Minkowski sum of a finite number of line segments, i.e.,
it has the form $K=[x_1,y_1]+\cdots + [x_n,y_n]$ with $x_i,y_i\in V$. In general, zonotopes are exactly the polytopes that have a center of symmetry such that all of its $2$ dimensional faces also have a center of symmetry; see~\cite[Theorem~3.5.2]{bible}.
The centrally symmetric zonotopes (i.e., with the center of symmetry at the origin)
are the ones that can be written in the form
\begin{equation}
 K =\tfrac{1}{2}[-x_1, x_1]+\cdots +\tfrac{1}{2}[-x_n, x_n].
\end{equation}
We introduce now the main objects of this paper.

\begin{definition}[Zonoids]
A convex body $K$ is called a \emph{zonoid} if it is the limit of a sequence of zonotopes
with respect to the Hausdorff metric.
\end{definition}

Our focus will be on the centrally symmetric zonoids.
Those are exactly the limits of centrally symmetric zonotopes.
We denote by $\ZZ(V)$ the space of centrally symmetric zonoids in $V$.\footnote{The standard notation
for centrally symmetric zonoids would be $\ZZ_0(V)$, but we omit the subscript in order to avoid unnecessary notation.}
By definition, $\ZZ(V)$ is a closed subspace of~$\KK(V)$.
It is known \cite[Cor. 3.5.7]{bible} that $\ZZ(V)$ is a proper subset
of the set of centrally symmetric
convex bodies iff $\dim V >2$.
We also note that if a polytope is a zonoid, then it must be a zonotope
(e.g., see~\cite{BLM_89}).
We abbreviate $\ZZ^m:=\ZZ(\R^m)$.

To deal with zonoids, we shall extensively use a probabilistic viewpoint
going back to Vitale~\cite{Vitale}, which not only is intuitive,
but also allows for a great deal of flexibility.

Let $\Omega $ be a probability space. By an \emph{integrable random convex body} we understand
a Borel measurable map $\y\colon\Omega \to \KK(V)$ such
that $\EE\|\y\|<\infty$;
see~\cref{eq:normB} for the definition of the norm.
The last condition implies that
$\EE h_{\y}(u)$ is a well defined sublinear function of $u$,
hence it is the support function of a uniquely defined convex body
(see \cref{useful_properties_for_h}(1)).
We can thus define the \emph{expectation} of~$\y$
to be the convex body $\EE\y\subset V$ with the following support function:
\be\label{eq:defEY}
 h_{\EE\y}(v) := \EE h_{\y}(v).
 \ee
Suppose now that $X$ is a random vector in $V$ satisfying
$\EE\|X\| < \infty$. We call such a random vector \emph{integrable}.
Then
$\y=\tfrac{1}{2}[-X,X]$ is an integrable, random, centrally symmetric segment.
Vitale~\cite{Vitale} proved that every  zonoid $K\subset V$ (including noncentered zonoids) can be written as $K=\EE [0,X] + c$ for some random integrable vector $X\in V$ and a fixed $c\in V$. We can write this as $K=\EE \tfrac{1}{2}[-X,X] + c + \tfrac{1}{2}\EE X$. If $K\in\Z(V)$ is centrally symmetric, then by uniqueness of support functions we have $h_K(u)=h_K(-u)$ for all $u\in V$, which implies $c + \tfrac{1}{2}\EE X=0$. Thus, we have shown that
\be\label{eq:K(X)-def}
 K(X) := \EE \tfrac{1}{2}\,[-X,X]
\ee
is a centrally symmetric zonoid,
and that every centrally symmetric zonoid arises this way.

\begin{definition}[Zonoid associated to random variable]\label{def:K}
If $X$ is an integrable random vector in $V$, we call $K(X)$ defined by \cref{eq:K(X)-def}
the \emph{Vitale zonoid associated to the random vector $X$}.
We also say that $K(X)$ \emph{is represented by~$X$}.
\end{definition}

In terms of support functions, we can describe $K(X)$ as follows.
\be\label{eq:hE-form}
  h_{K(X)}(v) = \EE h_{\tfrac{1}{2}[-X,X]}(v) = \tfrac{1}{2}\, \EE |\langle v, X \rangle| .
\ee
(This follows from \cref{eq:defEY} using that
$h_{[-z,z]}(v) = |\langle v, z\rangle|$.)
The observation $K(X)=K(-X)$ shows that the random variable
representing a zonoid is not unique.
However, it turns out that the random variable representing a zonoid is unique up to sign,
if one assumes $X$ to take its values
on the unit sphere $S(V)$.
This follows from the measure theoretic interpretation
of zonoids, we will discuss this point of view in \cref{sec:VZandmeas}  below (see also \cite[Theorem 3.5.3]{bible} and \cite[Theorem 5.2]{ACOCB}).

We found it of great technical advantage
to allow the random vector~$X$ to take values outside the unit sphere of $V$;
the resulting loss in uniqueness is not significant.
The next result is a clear indication of the advantage of this viewpoint.
It says that the association of a zonoid to an integrable random vector
commutes with linear maps.
Proving the following proposition using the measure point of view is involved (see \cref{continuity_of_linear} below), while the proof using random vectors is almost trivial. This highlights for the first time one of the advantages when working with random vectors.

\begin{proposition}\label{lemma:linear}
Let $V,W$ be Euclidean spaces,
$M:V\to W$ be a linear map and $X$ be an integrable random vector in~$V$.
Then $M(K(X))=K(MX)$.
\end{proposition}

\begin{proof}
We have for $v\in V$ by \cref{eq:hE-form},
$$
 h_{K(M(X))}(v) = \tfrac12 \EE \vert \langle v, MX \rangle \vert
 = \tfrac12 \EE \vert \langle M^T v, X \rangle \vert
 = h_K(M^T v) = h_{M(K)}(v),
$$
where last equality is due to
\cref{useful_properties_for_h}(4).
\end{proof}

\cref{lemma:linear}
also implies that linear images of zonoids are again zonoids.
Of course, this also follows directly from the definition.

For instance, we obtain for a standard Gaussian vector $X\in \R^m$,
\be\label{eq:ball}
 K(X)=(2\pi)^{-\frac{1}{2}}B^m,
\ee
where $B^m$ denotes the Euclidean unit ball in $\R^m$.
Indeed,
$
h_{K(X)}(v)=\tfrac{1}{2}\EE|\langle v, X\rangle|
  =(2\pi)^{-\frac{1}{2}}\, \|v\|,
$
by \cref{eq:hE-form},
since
$\sum_{i=1}^m X_i v_i$ is distributed as $\|v\|\, Z$, where $Z\in\R$ is standard gaussian
and has expected value $\EE|Z|=\sqrt{2/\pi}$.
The function $h_{K(X)}$
equals the support function of $(2\pi)^{-\frac{1}{2}}B^m$
and hence~\cref{eq:ball} follows.

\begin{remark}\label{rk:lolnforzonoids}
Behind Vitale's result is the following law of large numbers for zonoids~\cite{ArtsteinVitale},
which provides a geometric way for viewing Vitale zonoids associated to random vectors.
Let $X$ be an integrable vector in $V$ and $\{X_k\}_{k\geq 1}$
a sequence of independent samples of $X$.
Then, as $n\to\infty$, we have in almost sure convergence
$
\tfrac{1}{n}\sum_{k=1}^n\tfrac{1}{2}[-X_k, X_k] \rightarrow K(X).
$
\end{remark}

We next show how to realize the Minkowski sum of two zonoids as the Vitale zonoid associated to a random vector.

\begin{lemma}\label{summation_formula}
Let $X,Y$ be integrable random vectors in $V$. Then $K(Z)=K(X)+K(Y)$
for the random variable $Z:=2\epsilon X + 2(1-\epsilon)Y$ defined with
a fair Bernoulli random variable $\epsilon\in\{0,1\}$ that is independent of $X$ and $Y$.
\end{lemma}

\begin{proof}
By \cref{eq:hE-form}, we have
$$
h_{K(Z)}(v) = \tfrac{1}{2}\EE\, h_{[-Z,Z]}(v)
 = \tfrac{1}{2} \EE_{X,Y}\EE_\epsilon\,\,h_{2 \epsilon [-X, X] + 2 (1-\epsilon) [-Y,Y]}(v).
$$
Expanding the expectation over $\epsilon$,
$$
\EE_\epsilon\,\,h_{\epsilon 2 [-X, X] + (1-\epsilon) 2 [-Y,Y]}(v) =
 \tfrac{1}{2}  h_{2 [-X, X]}(v)+ \tfrac{1}{2}  h_{2 [-Y,Y]}(v)
  = h_{ [-X, X]}(v)+ h_{[-Y,Y]}(v)
$$
and taking expectations over $X,Y$ finishes the proof.
\end{proof}

\cref{summation_formula} has a straightforward generalization for adding
$n$~random variables.
As an application, if $x_1, \ldots, x_n\in V$ are fixed vectors
and $Z\in V$ is the random vector taking the value $n x_i$ with probability $1/n$,
we see that
$K(Z)$ is the zonotope $\tfrac{1}{2}[-x_1, x_1]+\cdots +\tfrac{1}{2}[-x_n, x_n]$.

The next observation states that scaling the
random variable~$X$ with an independent random function only
leads to a rescaling of the resulting zonoid.

\begin{lemma}\label{lemma:scaling}
We have
$K(\rho X)=\EE|\rho|\cdot K(X)$
if $X\in V$ and $\rho \in \R$ are
independent integrable random variables.
\end{lemma}

\begin{proof}
We have
$h_{K(\rho X)}(v) = \tfrac{1}{2}\EE |\langle v, \rho X\rangle|
 = \tfrac{1}{2} \EE\, \big(|\rho|\cdot |\langle v, X\rangle| \big)
 =\EE|\rho|\cdot h_{K(X)}(v)$,
where the first equality is~\cref{eq:hE-form} and
the last equality is due to the independence of~$X$ and $\rho$.
The assertion follows since convex bodies are determined by their support function.
\end{proof}

As an application, we compute the zonoid defined by  $\widetilde{X}$ uniformly distributed
on the unit sphere $S^{m-1}$ in $\mathbb R^m$.
For this, we introduce the following notation:
\be\label{eq:rhodef}
\tau_m:= \sqrt{2\pi} \mean \Vert X\Vert, \; \text{ where $X\in\R^m$ is a standard Gaussian random vector}.
\ee
This number is $\sqrt{2\pi}$ times the expected value of a chi-distributed random variable with
$m$~degrees of freedom\footnote{In \cite{PSC} the expected value of a chi-distributed random variable with $m$ degrees of freedom is denoted~$\rho_m$. So, using their notation we have $\tau_m=\sqrt{2\pi}\,\rho_m$. The additional factor of $\sqrt{2\pi}$ makes the formulas in this paper easier to grasp.}
It has the explicit value $\tau_m:=\sqrt{2\pi}\, \sqrt{2}\,\Gamma\left(\frac{m+1}{2}\right)  / \Gamma\left(\frac{m}{2}\right)$. We can write
$\widetilde{X} := X/\|X\|$, where $X\in\R^m$
is standard Gaussian. Then, we have
\be\label{eq:unif-sphere}
 K(\widetilde{X})
 = \frac{1}{\tau_m \sqrt{2\pi}}\; B^m,
\ee
where, as before, $B^m$ is the unit ball in $\R^m$.
This follows from \cref{lemma:scaling} and \cref{eq:ball},
using that $\widetilde{X}$ is independent of $\|X\|$ and (e.g., see~\cite[\S 2.2.3]{condition})

In the next example we generalize this  from spheres to Grassmannians.

\begin{example}\label{ex:KX}
Let $\xi_1, \ldots, \xi_k\in \R^m$ be standard, independent Gaussian vectors and consider
\be
X:=\xi_1\wedge \cdots \wedge \xi_k\in \Lambda^k(\R^m).
\ee
Put $\widetilde{X}:=X/\|X\|$.
By \cite[Theorem 8.1]{James}, the random variables $\widetilde{X}$ and $\|X\|$ are independent
and $\widetilde{X}$ is uniformly distributed on
the Grassmannian\footnote{Recall that $\widetilde{G}(k,m)$ can be identified with the set of unit simple vectors in $\Lambda^k(\R^m)$, see \cite{Kozlov}.}
$\widetilde{G}(k,m)$ of \emph{oriented} $k$--planes in~$\R^m$.
Moreover, $\|X\|$ is distributed as the square root of the determinant of a Wishart matrix\footnote{Recall that, if $M=[\xi_1, \ldots, \xi_k]$ is a $\R^{m\times k}$ matrix whose columns are standard independent gaussian vectors in~$\R^m$, the matrix $A:=M^TM$ is called a Wishart matrix.}.
In particular, using \cite[Theorem 3.2.15]{Muirhead}, we get
$
\EE\|X\|=2^{\frac{k}{2}}\,\Gamma_{k}\left(\frac{m+1}{2}\right) / \Gamma_{k}\left(\frac{m}{2}\right),
$
where $\Gamma_k(x):=\pi^{\frac{k(k-1)}{4}}\prod_{j=1}^k \Gamma\big(x+\tfrac{1-j}{2}\big)$ denotes the  \emph{multivariate Gamma function}; see, e.g., \cite[Theorem 2.1.12]{Muirhead}.
\cref{lemma:scaling} tells now that
\be\label{eq:zonograss}
K(X)=2^{\frac{k}{2}}\frac{\Gamma_{k}\left(\frac{m+1}{2}\right)}{\Gamma_{k}\left(\frac{n}{2}\right)}K(\widetilde{X}).
\ee
Notice that, in the case $k=1$, we get $\EE\|X\|=\tau_m/\sqrt{2\pi}$ (see \cref{eq:rhodef}) and we recover~\cref{eq:unif-sphere}
as a special case of \cref{eq:zonograss}.
\end{example}

We now show that the expectation of the norm of an integrable random vector~$X$
depends only on its zonoid $K(X)$.
Recall that $\bar{h}_K$ denotes the restriction of $h_K$ to
the unit sphere.

\begin{proposition}\label{lemma:ideno}
Let $X$ be an integrable random vector in $\R^m$ and $K=K(X)$.
Then
$$
 \EE\|X\| = \sqrt{2\pi}\, \EE h_{K}(v)
   = \tau_m\, \|\bar{h}_K \|_{1}
$$
where $v\in\R^m$ is a standard Gaussian vector, and
the $1$-norm is as in \cref{L_norms}.
\end{proposition}

\begin{proof}
We assume w.l.o.g.\ that $v$ is independent of $X$.
By rotational invariance and homogeneity, we have
$\|x\| = \sqrt{\pi/2}\,\EE|\langle x, v \rangle|$ for all $x\in \R^m$.
This implies, using the independence of $X$ and $v$, that
\begin{align}\label{eq:lengthbyGauss2}
    \EE\|X\|=\sqrt{\frac{\pi}{2}}\,\underset{X,v}{\EE}\left|\langle X, v\rangle\right|=\sqrt{\frac{\pi}{2}}\,\underset{X,v}{\EE}\,2\cdot h_{\tfrac{1}{2}[-X,X]}(v)
     =\sqrt{2\pi}\cdot\EE_v\, h_{K}(v),
\end{align}
which shows the first equality.
The second equality follows from
the factorization $v = \|v\|~\cdot~\tilde{v}$,
noting as before that $\|v\|$ and $\tilde{v}$,
are independent and $\tilde{v}$ is uniformly distributed on the unit sphere.
We have $\EE \|v\| = (2\pi)^{-\frac{1}{2}}\,\tau_m $  and
$\EE_v\, h_{K}(v) = \Vert \bar{h}_K\Vert_{1}$ by~\eqref{L_norms},
because $\bar{h}_K$ is nonnegative.
Putting everything together finishes the proof.
\end{proof}

\cref{lemma:ideno}
shows that the following notion of length is well defined.
The length will be investigated more closely in \cref{sec:MVandL}.

\begin{definition}\label{mydef:length} 
We define the \emph{length} of a zonoid $K\in \Z(V)$ by
$$
 \ell(K)=\EE\|X\|,
$$
where $X$ is an integrable random vector representing $K$.
\end{definition}

\begin{example}
	We see by \cref{lemma:ideno} that if $B^m$ is the unit ball of $\R^m$
	then its length is given by $\ell(B^m)=\tau_m$,
	where $\tau_m$ is given by~\cref{eq:rhodef}.
\end{example}

We next show that
the length is additive with respect to Minkowski addition.
This might be surprising at first sight,
given that~\cref{mydef:length} defines the length as the expected value of a norm.

\begin{lemma}\label{length_is_additive}
  Let $K,L\in\Z(V)$ and $\lambda \geq 0$. Then, we have
  $$\ell(K+\lambda L) = \ell(K) + \lambda \ell(L).$$
\end{lemma}

\begin{proof}
Let $K=K(X)$ and $L=K(Y)$ for independent integrable vectors $X,Y\in V$.
By \cref{lemma:scaling} we have $\lambda L = K(\lambda Y)$.
Using this and \cref{summation_formula} we can write the Minkowski sum as
$K+\lambda L = K(2\epsilon X + 2\lambda(1-\epsilon)Y)$,
where $\epsilon$ is a fair Bernoulli random variable $\epsilon\in\{0,1\}$
that is independent of $X$ and $Y$. This implies
$\ell(K+\lambda L) = \mean \Vert 2\epsilon X + 2\lambda(1-\epsilon)Y\Vert$.
Taking first the expectation over $\epsilon$ yields
$\ell(K+\lambda L) = \tfrac{1}{2}\mean \Vert  2X\Vert + \tfrac{1}{2}\mean \Vert 2\lambda Y\Vert =  \ell(K) + \lambda \ell(L)$.
This finishes the proof.
\end{proof}

Combining \cref{lemma:ideno} with \cref{useful_properties_for_h}(3) we get the following corollary.
\begin{corollary}[Monotonicity of the length]\label{cor:monolength}
Let $K, L \subset V$ are zonoids such that $K\subset L$, then $\ell(K)\leq\ell(L)$.
\end{corollary}

We recall from \cref{useful_properties_for_h} and \cref{lemma:ideno} that
the radius $\Vert K\Vert$ of a zonoid $K$ can be expressed as the $\infty$--norm
of its support function $\bar{h}_K$ and that its length $\ell(K)$ can expressed
in terms of its $1$--norm.
The next result compares these two norms.

\begin{corollary}[Radius and length]\label{cor:ideno}
We have
$$
  2\|K\| \ \le\ \ell(K) \ \le\  \tau_m\,  \|K\| ,
$$
with inequality on the left hand side iff $K$ is a (centrally symmetric) segment,
and equality holding on the right hand side
iff $K$ is rotational invariant.
\end{corollary}

\begin{proof}
Recall $\|K\| = \|\bar{h}_K\|_{\infty}$ from \cref{useful_properties_for_h} and $\ell(K) = \tau_m \|\bar{h}_K \|_{1}$ from \cref{lemma:ideno}.
By definition of the norms we have $\|h_K\|_{1} \le \|h_K\|_{\infty}$ with equality holding iff
$h_K$ is constant. The latter means that $K$ is rotationally invariant.
This shows the right inequality.

For the left inequality, we write $K=K(X)$. By \cref{eq:hE-form} we have
$2 h_K(u) = \EE |\langle X, u \rangle| \le \EE \|X\|$,
if $\|u\|=1$. This implies $2\|\bar{h}_K\|_{\infty} \leq  \EE \|X\| = \ell(K)$
and equality holds if $K$ is a segment.
\end{proof}

The next observation will be useful later.

\begin{lemma}\label{le:orthogonal}
Let $K,L\in Z(V)$ be zonoids and $\spanK{K},\spanK{L}$ denote their linear spans.
If $K,L$ are represented by the integrable random vectors $X,Y$,
respectively, then
$$
 \spanK{K} \perp \spanK{L} \Longleftrightarrow \langle X, Y \rangle = 0 \: \mbox{almost surely} .
$$
\end{lemma}

\begin{proof}
Let $\pi\colon V \to \spanK{L}$ denote the orthogonal projection.
By \cref{useful_properties_for_h} and \cref{eq:hE-form},
the support function of $\pi(K)$, for $u\in \spanK{L}$, is given by
$$
 h_{\pi(K)}(u) = h_K(\pi(u)) = \tfrac{1}{2}\EE|\langle \pi(u), X\rangle|  .
$$
Thus $\spanK{K} \perp \spanK{L}$ iff  $\pi(K)=0$ iff $h_{\pi(K)}=0$.
The latter is equivalent to
$\langle \pi(u), X\rangle = 0$ almost surely,
for all $u\in \spanK{L}$. This is easily seen to be equivalent to
$\langle X, Y \rangle = 0$ almost surely.
\end{proof}

\subsection{Virtual zonoids}\label{sec:virtualzonoids}

It is well known that the set of zonoids~$\Z(V)$ and convex bodies ~$\K(V)$ can be interpreted as cones in
vector spaces of ``virtual zonoids'' $\VZ(V)$ and  ``virtual convex bodies'' $\VK(V)$, respectively (see \cite[\S 25.1]{BuZa}) and \cite[\S 3.5]{bible}).
This allows to investigate them using tools from linear algebra and functional analysis.

We confine ourselves to zonoids, since only for those
we can define a satisfying notion of tensor product, see \cref{se:tensorproduct}.
The next result summarizes the situation.

\begin{theorem}\label{th:virt:zonoid}
$\Z(V)$ is embedded in an essentially unique way as a subcone of a normed and partially ordered
\emph{real vector space}~$\VZ(V)$
of \emph{virtual zonoids}
such that any element of $\VZ(V)$ can be written as a formal difference
$K_1 -K_2$ of zonoids. The norm and partial order are defined as follows:
$$
 \|K_1 -K_2\| = d_H(K_1,K_2), \qquad
 0 \le K_1 - K_2\ :\Longleftrightarrow\ h_{K_2} \le h_{K_1} .
$$
Thus the norm extends the norm of zonoids introduced in~\cref{eq:normB}
and the partial order on $\VZ(V)$  extends the inclusion of convex bodies (see \cref{useful_properties_for_h}).
However, $\VZ(V)$ is not complete, so not a Banach space, unless $\dim V \le 1$,
\end{theorem}

\begin{proof}
We can identify convex bodies with their support functions
using \cref{useful_properties_for_h}. Then $\K(V)$ can be seen as
a real cone in $C(S)$, which is closed under addition and multiplication with nonnegative scalars
(recall that $S$ denotes the unit sphere in $V$).
We define the space of virtual convex bodies $\VK(V)$ as the span of~$\K(V)$:
it is the subspace of $C(S)$ consisting of differences of support functions.
Similarly, we define the space of virtual zonoids $\VZ(V)$
as the span of $\Z(V)$, seen as a subset of $C(S)$.
Thus $\VZ(V)\subset\VK(V)$ are endowed with the supremum norm
(of functions restricted to the unit sphere),
which extends the norm of
convex bodies defined in~\cref{eq:normB}. Moreover, the pointwise order of functions
makes them partially ordered vector spaces.
For the uniqueness, we refer to the discussion below.
The non-completeness of $\VZ(V)$ is shown in
\cref{propo:banach} below.
\end{proof}

In \cref{sec:VZandmeas} we will see that virtual zonoids can be identified with the space of signed measures on the sphere. Moreover, virtual zonoids also have a geometric realization as \emph{hedgehogs}, i.e., envelopes of hyperplanes defined by the difference of the corresponding support functions. For this point of view we refer the reader to \cite{HZ}.

Despite the efficiency of support functions,
used in the proof of \cref{th:virt:zonoid}, it is useful
to think of virtual convex bodies more abstractly
as formal differences of convex bodies.
The abstract point of view reveals the uniqueness of the construction.
Let us therefore formulate the above algebraic features in this language.

Given a commutative monoid $(\mathbb{M}, +)$, 
its \emph{Grothendieck group} $(\widehat{\mathbb{M}}, +)$ can be defined as the group
of equivalence classes of pairs $(m_1, m_2)\in \mathbb{M}$ such that $(m_1, m_2)\sim (n_1, n_2)$
if and only if $m_1+n_2=m_2+n_1$,
and with the addition $[(m_1, m_2)]+[(n_1, n_2)]:=[(m_1+n_1, m_2+n_2)]$.
The Grothendieck group of $(\mathbb{M}, +)$ is an essentially unique object characterized by a universal property,
see \cite{lang-algebra}.
If the cancellation law holds, then $\mathbb{M}\to \widehat{\mathbb{M}},\, m\mapsto [(m,0)]$
is injective so that $\mathbb{M}$ can be seen as a submonoid of $\widehat{\mathbb{M}}$
and we write the class $[(m_1, m_2)]$ as $m_1-m_2$.
If, in addition, there is an action of $\R_+$ on $\mathbb{M}$ satisfying the usual axioms of scalar multiplication in
vector spaces, then it is immediate to check that the scalar multiplication extends to~$\R$, so that
$\widehat{\mathbb{M}}$ becomes a real vector space, which we may call the corresponding
\emph{Grothendieck vector space}.

We apply this construction to the monoids of zonoids and convex bodies in a Euclidean space~$V$, endowed
with the Minkowski addition and the scalar multiplication given by \cref{scalar_multiplication}.
Using support functions, we see that the cancellation law holds in $(\KK(V), +)$,
see \cref{useful_properties_for_h} and \cite[\S 3.1]{bible}.

\begin{definition}[The vector space of virtual convex bodies]
We denote by $\VK(V)$ and $\VZ(V)$
the Grothendieck vector spaces of $\mathcal{K}(V)$
and $\ZZ(V)$, respectively.
Its elements are called
\emph{virtual convex bodies} and \emph{virtual zonoids}, respectively.
\end{definition}

From the uniqueness of the Grothendieck vector spaces $\VK(V)$ and $\VZ(V)$ up to
canonical isomorphy, it follows that they are isomorphic
to the subspaces of $C(S)$ defined via support functions
in the proof of \cref{th:virt:zonoid}.
We abbreviate $\VK^m := \VK(\R^m)$ and $\VZ^m := \VZ(\R^m)$.

Let us point out that
the formal inverse $-K$ of a zonoid $K$, which is a virtual zonoid,
should not be confused with the zonoid $(-1)K$, which is equal to $K$.

\begin{remark}\label{re:nn-cone}
1. We may define the nonegative cone of $\VK(V)$ corresponding to the above defined order by
$
 \VK(V)_+ := \big\{ K_2 -K_1 \mid K_1, K_2 \in \K(V), K_1 \le \K_2 \big\} ;
$
similarly, we define $\VZ(V)_+$. Then it is clear that $\K(V) \subseteq \VK(V)_+$
and $\Z(V) \subseteq \VZ(V)_+$, however,
the inclusions in general are strict~\cite[\S 3.2]{bible}.
\end{remark}

\begin{remark} The elements of $\KK(V)\cap \VZ(V)$ are the convex bodies that can be written
as a difference of zonoids, they are called \emph{generalized zonoids} \cite[p.\ 195]{bible}.
They are the convex bodies whose support function has a representation
as in \cref{eq:ct} below, but with an even signed measure~$\mu$.
It is known that the generalized zonoids are dense in $\K(V)$;
in particular, there exist centrally symmetric convex bodies,
which are not generalized zonoids. This is shown in~\cite[Note 13, p.\ 206]{bible}.
\end{remark}

If $M:V\to W$ is a linear map between Euclidean vector spaces then $K\mapsto M(K)$
gives a morphism $\ZZo(V) \to \ZZo(W)$ of monoids, which commutes with multiplication with nonnegative scalars.
It is immediate that this morphism can be extended to a linear map
between the corresponding Grothendieck vector spaces.

\begin{definition}\label{def:linearext}
Let $M\colon V\to W$ be a linear map between Euclidean vector spaces.
We denote by $\widehat{M}:\VZ(V)\to \VZ(W)$ the \emph{associated linear map}
defined by 
\be
 \widehat{M}(K_1-K_2) := M(K_1) - M(K_2).
 \ee
\end{definition}


We collect some properties of the associated map $\widehat{M}$.

\begin{proposition}\label{propo:linearext}
Let $M:V\to W$ be a linear map between Euclidean vector spaces.
Then the associated linear map $\widehat{M}:\VZ(V)\to \VZ(W)$ 
preserves the order, is continuous, and
$$
 \|\widehat{M}\|_{\mathrm{op}}=\|M\|_{\mathrm{op}} .
$$
Moreover, $\widehat{M}(\ZZo(V))\subseteq \ZZo(W)$,
i.e., $\widehat{M}$ maps zonoids to zonoids.
Finally, for $K\in \ZZo(V)$,
$$
\ell(M(K)) \ \le\  \|M\|_{\mathrm{op}}\, \ell(K) .
$$
In particular, the length of zonoids does not increase under an orthogonal projection.
\end{proposition}
\begin{proof}
The preservation of the order follows from the definition and \cref{useful_properties_for_h}.
To prove continuity, let $Z=K_1-K_2\in \VZ(V)$.
Recall from \cref{th:virt:zonoid}  that the norm defined on $\VZ(V)$ is $\|Z\|=d_H(K_1, K_2)$.
Using the properties from \cref{useful_properties_for_h}, we have
\begin{align}\|\widehat{M}(Z)\|&= \Vert M(K_1)-M(K_2)\Vert \\
&=   \sup\{|h_{M(K_1)}(u)-h_{M(K_2)}(u) \mid u\in W,\ \|u\|=1\}  \\
&=   \sup\{|h_{K_1}(M^T u)-h_{K_2}(M^T u)| \mid u\in W,\ \|u\|=1\}  \\
&\leq\Vert M^T\Vert_{\mathrm{op}}\cdot \sup\{| h_{K_1}(v)-h_{K_2}(v) | \mid v\in W,\ \|v\|=1\}\\
&= \|M^T\|_{\mathrm{op}}\cdot d_H(K_1,K_2) \\
&= \|M\|_{\mathrm{op}}\cdot \|K_1-K_2\|= \|M\|_{\mathrm{op}}\cdot \|Z\|.
\end{align}
This proves
$\|\widehat{M}\|_{\mathrm{op}}\leq \|M\|_{\mathrm{op}}$
and hence the continuity of the linear map $\widehat{M}$.
On the other hand, let $v\in V$ be of norm one such that $\|M\|_{\mathrm{op}}=\|Mv\|$
and consider the segment $\frac{1}{2}[-v,v]$,
of norm $\|\frac{1}{2}[-v,v]\|=\|v\|=1$.
Then
\begin{align}\|\widehat{M}\|_{\mathrm{op}}\geq  \left\|\widehat{M}\left(\tfrac{1}{2}[-v,v]\right)\right\|= \left\|\left(\tfrac{1}{2}[-Mv,Mv]\right)\right\|=\|Mv\|=\|M\|_\mathrm{op} ,
\end{align}
This proves the stated equality of operator norms.

Since $\widehat{M}$ maps segments to segments and $\widehat{M}$ is continuous,
we have $\widehat{M}(\ZZo(V))\subseteq \ZZo(W)$.
For the stated upper bound on the length, let $X$ be a integrable random vector
representing~$X$. Then $M(X)$ represents the image $M(K)$ and we have by
\cref{mydef:length},
$$
 \ell(M(K)) = \EE \|M(X)\| \ \le\  \|M\|_{\mathrm{op}}\,\EE \|X\|
  = \|M\|_{\mathrm{op}}\,\ell(X) ,
$$
which completes the proof.
\end{proof}

Our construction may be summarized by saying that
we have constructed a covariant functor from the category of
Euclidean spaces (with any linear morphisms)
to the category of partially ordered and normed vector spaces.

The length of zonoids defined in \cref{mydef:length}
extends to a continuous linear functional.

\begin{theorem}\label{th:lengthLF}
The length extends to a continuous linear functional
$\ell\colon \VZ(V) \to \R$.
\end{theorem}
\begin{proof}
\cref{length_is_additive} shows that the length on $ \VZ(V)$ is additive, so we can extend it to a linear functional on $\VZ(V)$ by setting $\ell(K-L):=\ell(K)-\ell(L)$. Using \cref{cor:ideno} we can characterize this as $\ell(K_1-K_2) = \tau_m \,(\Vert \bar{h}_K\Vert_{1}- \Vert\bar{h}_L\Vert_{1})$.
Since both $\bar{h}_K$ and $\bar{h}_L$ are nonnegative we have $\Vert \bar{h}_K\Vert_{1} - \Vert\bar{h}_L\Vert_{1} = \Vert \bar{h}_K - \bar{h}_L\Vert_{1}$. This shows that the length is continuous
\end{proof}

We also need the following observation.
\begin{lemma}\label{re:1-dim}
If $V$ is one-dimensional, the length $\ell: \VZ(V)\to \R$
induces an isomorphism of additive groups, which
preserves the standard norm and the standard order.
We shall use this to identify $\VZ(V)$ with $\R$ and $\Z(V)$ with $\R_+$.
\end{lemma}
\begin{proof}
The length is a group homomorphism by \cref{th:lengthLF}.
Let $K=\tfrac{1}{2}[-x,x]$ be a centrally symmetric zonoid in $\ZZ(V)$; i.e., a segment.
Then, $\ell(K) = x$, which shows that $\ell$ is bijective, hence an isomorphism. This also shows that $\ell$ preserves the standard norm and the standard order.
\end{proof}

We finally show that $\VZ(V)$ in general is not a Banach space,
as already claimed in \cref{th:virt:zonoid}.
Recall that $C(S)$ denotes the Banach space of continuous
real functions on the unit sphere $S:=S(V)$ endowed with the supremum norm.
We denote by $C_{\mathrm{even}}(S)$ its closed subspace of
even functions (i.e., $f(-v) =f(v)$).
In the proof of \cref{th:virt:zonoid}
we have constructed the normed vector space of virtual convex bodies
$\VK(V) \subset C(S)$
and the normed vector space $\VZ(V) \subset C_{\mathrm{even}}(S)$ of virtual zonoids.
(Recall that the sup-norm corresponds to the Hausdorff metric~\cref{def_Hausdorff_metric}.)
Unfortunately, these are not Banach spaces.
This follows from known facts that we state in the next proposition
for the sake of clarity.

\begin{proposition}\label{propo:banach}
The completion of the space of virtual convex bodies $\VK(V)$ equals $C(S)$
and the completion of the space of virtual zonoids $\VZ(V)$ equals~$C_{\mathrm{even}}(S)$.
Moreover, if $\dim V > 1$, then $\VK(V) \ne C(S)$ and
$\VZ(V) \ne C_{\mathrm{even}}(S)$, hence
$\VK(V)$ and $\VZ(V)$ are not complete and hence not Banach spaces.
In particular, they are real vector spaces
of infinite dimension.
However, if $\dim V = 1$, then we have~$\VK(V)\simeq \R^2$ and $\VZ(V)\simeq \R$.
\end{proposition}

\begin{proof}
W.l.o.g.\ we assume $V=\R^m$. It is known~\cite[Lemma 1.7.8]{bible} that every twice continuously differentiable
function $S^{m-1}\to\R$
can be written as a difference $h_K - r h_{B^m}$ for some $K\in\KK^m$ and $r>0$.
This implies that $\VK^m$ is dense in $C(S^{m-1})$, since
continuous function can be uniformly approximated by twice continuously differentiable functions.
On the other hand, $\VK^m$ is strictly contained in $C(S^{m-1})$ if $m>1$. One one way to see this
is that there are continuous, nowhere differentiable functions on $S^{m-1}$, while
every $h\in \VK^m$ is differentiable at least at some point in $S^{m-1}$.

The assertion about virtual zonoids is more involved. In~\cite[Theorem 3.5.4]{bible}, it is shown that
every sufficiently smooth even real function $f\colon S^{m-1}\to\R$ has a representation
as in \cref{def_cosine_transform} below, but with a even signed measure~$\mu$, which implies that
$f\in \VZ^m$. This implies that $\VZ^m$ is dense in $C_{\mathrm{even}}(S^{m-1})$.
In order to show that $\VZ^m$ is strictly contained in $C_{\mathrm{even}}(S^{m-1})$
one can argue as follows.
If we had $\VZ^m=C_{\mathrm{even}}(S^{m-1})$, then every centrally symmetric convex body
would be a virtual zonoid. This contradicts \cite[Note 13, p.\ 206]{bible},
which states that there exist centrally symmetric convex bodies,
which are not generalized zonoids.
Finally, $\VK^1\simeq \R^2$ and $\VZ^1\simeq \R^1$ since
$\K^1$ consists of the intervals and $\Z^1$ consist of the symmetric intervals.
\end{proof}

We remark that the fact that the cone  $\Z(V)$ is a closed subset of the cone $\K(V)$
does not contradict the fact that, by \cref{propo:banach},
$\VZ(V)$ is a dense subset of $\VK(V)$.

\subsection{Virtual zonoids and measures}\label{sec:VZandmeas}
There is a correspondence between zonoids and even measures on the sphere, that we discuss now. This point of view is classic when dealing with zonoids, so we include here a review of the principal facts of this correspondence. The approach using measures provides a complimentary viewpoint to our approach using random vectors. This alternative viewpoint will become particularly useful in \cref{sec:continuity}, where we discuss continuity properties of our constructions.

In what follows we identify the space of even measures on the sphere with the space of measures on the projective space. The space of continuous functions on the projective space is $C(\PP^{m-1})$. The space of (signed) measures on $\PP^{m-1}$ will be denoted by $\mathcal{M}(\PP^{m-1})$. The cone of positive measures will be denoted by $\mathcal{M}_+(\PP^{m-1})$.
Recall that the weak--$*$ topology on $\cM(\PP^{m-1})$ is the coarsest topology on $\mathcal{M}(\PP^{m-1})$ such that for every $\phi\in C(\PP^{m-1})$ the linear functional $\mu\mapsto \int_{\PP^{m-1}}\phi(x)\mu(\mathrm{d}x)$ is continuous.

For practical purposes,  an even measure on the sphere $S^{m-1}$ will correspond to $\frac{1}{2}$ times its pushforward measure on the projective space. Similarly we identify a function on the projective space with the corresponding even function on the sphere; that is, if $f:\PP^{m-1}\to\R$, then we write $f(x):=(f\circ \Pi)(x)$, where
$\Pi: \R^m\setminus\{0\} \to \mathbb P^{m-1}$ is the projection.

With these identifications, for all $\mu\in\mathcal{M}(\PP^{m-1})$ and $f\in C(\PP^{m-1})$, we have
\be
    \int_{\PP^{m-1}} f(x) \ \mu(\dd x) = \frac{1}{2} \int_{S^{m-1}}f(x) \ \mu(\dd x).
\ee

Passing from zonoids to virtual zonoids corresponds to passing from even measures to even signed measures. The main object of the section is the following map.

\begin{definition}\label{def_cosine_transform}
The \emph{cosine transform} is the linear map $\HH:\mathcal{M}(\PP^{m-1})\to C(\PP^{m-1})$ given for all $\mu \in \mathcal{M}(\PP^{m-1})$ and $x\in\PP^{m-1}$ by
\be\label{eq:ct}
    \HH(\mu)(x):=\int_{\PP^{m-1}} |\langle u, x\rangle|\ \mu(\dd u).
\ee
The image of the cosine transform will be denoted by $\HH(\PP^{m-1})$ and the image of $\mathcal{M}_+(\PP^{m-1})$ by $\HH_+(\PP^{m-1})$.
\end{definition}

We could not locate any reference for items (4)--(6) of the next theorem. This is why we include our own proof.  Recall that we endow $\mathcal{M}(\PP^{m-1})$ with the weak--$*$ topology and $C(\PP^{m-1})$ with the topology induced by the $\infty$--norm.

\begin{theorem}\label{th:mainZ}
The cosine transform
$\HH\colon\mathcal{M}(\PP^{m-1}) \to C(\PP^{m-1})$ satisfies the following properties.
\begin{enumerate}
\item $\HH$ is injective.
\item  $\HH(\mathcal{M}(\mathbb P^{m-1}))$ is a dense subspace of $C(\PP^{m-1})$.
\item There is $c=c(m)>0$ such that
$c\Vert \HH(\mu)\Vert_\infty \le \mu(\PP^{m-1}) \le \Vert \HH(\mu)\Vert_\infty$
for all (nonnegative) measures $\mu\in \mathcal{M}_+(\PP^{m-1})$.
\item $\HH$ is sequentially continuous.
\item The restriction $\HH\colon \mathcal{M}_+(\PP^{m-1}) \to \HH_+(\PP^{m-1})$
is a homeomorphism.
\item The inverse $\HH^{-1}: \HH(\PP^{m-1}) \to \mathcal{M}(\PP^{m-1})$ is not sequentially continuous for $m>1$.
\end{enumerate}
\end{theorem}
\begin{remark}
Our proof for item (4) does not extend to nets. This is why we only prove sequential continuity of $\HH$.
\end{remark}
\begin{proof}[Proof of \cref{th:mainZ}]
Assertions (1) and (2) are in (the proof of) \cite[Theorem 3.5.4]{bible}.
Assertion (3) follows from \cref{cor:ideno} stated in our context, and using \cref{remark_measures} (2) below.

As for assertion~(4): because $\HH$ is linear it suffices to prove sequential continuity at $0$. Recall that for the weak--$*$ topology on $\cM(\PP^{m-1})$, a sequence of measures $(\mu_i)$ converges to $0$, if and only if for every $\phi\in C(\mathbb P^{m-1})$, the sequence $\langle \mu_i,\phi \rangle$ goes to zero, where
\begin{equation}\label{def_inner_product}
\langle \mu, \phi\rangle:=\int_{\mathbb P^{m-1}}\phi(x)\;\mu(\mathrm{d} x).
\end{equation}
Suppose that $(\mu_i)$ converges to $0$ in $\cM(\PP^{m-1})$ in the weak--$*$ topology.
Let $h_i := \cH({\mu_i})$, thus
$
h_i(v) := \int_{x\in\PP^{m-1}} |\langle v,x \rangle | \, \mu_i(\dd x),
$
and in particular $h_i(v) \to 0$ for all $v\in\PP^{m-1}$.
So we have pointwise convergence of the $h_i$.
We are going to show that $h_i \to  0$ uniformly. Recall that every measure $\mu$ has a unique decomposition $\mu=\mu^+-\mu^-$ (called the Hahn--Jordan decomposition), where $\mu^+, \mu^-\in\mathcal M_+(\mathbb P^{m-1})$. We define $\vert \mu\vert :=\mu^++\mu^-$.
The Banach--Steinhaus Theorem (e.g., see \cite{rudin:73})
implies that
$
\kappa := \sup_i (|\mu_i|(\PP^{m-1})) < \infty .
$
Therefore, we have for any $v\in\PP^{m-1}$,
$$
|h_i(v)| \le \int_{\PP^{m-1}} |\langle v, x\rangle |\,  |\mu_i|(\dd x) \le |\mu_i|(\PP^{m-1}) \le \kappa ,
$$
and hence $\sup_i \|h_i\|_{\infty} \le \kappa$.
Moreover, for $v_1,v_2 \in\PP^{m-1}$,
$$
|h_i(v_1) - h_i(v_2) | \le \int_{\PP^{m-1}} | \langle v_1 - v_2 , x\rangle | \, |\mu_i| (\dd x)
\le \kappa \|v_1 -v_2\| .
$$
The Arzel\`{a}-Ascoli Theorem (e.g., see \cite{rudin:73}) implies that
$(h_i)$ has a uniformly convergent subsequence $(h_{i_j})$.
Thus $h_{i_j} \to 0$ uniformly since $h_{i_j} \to 0$ pointwise. By the same argument we see that any subsequence of $(h_i)$ has a subsequence
that uniformly converges to $0$. This implies that $h_i \to 0$ uniformly. Therefore, it follows
that the map $\cH$ is sequentially continuous.

For assertion (5),
Bolker~\cite[Theorem 5.2]{ACOCB} showed
that $\cH\colon\mathcal{M}_+(\PP^{m-1}) \to C(\PP^{m-1})$ is continuous. So we only need to show that the inverse $\HH^{-1}: \HH_+(\mathbb P^{m-1})\to \mathcal{M}_+(\mathbb P^{m-1})$ is continuous. For this it suffices to show that $\HH^{-1}$ is sequentially continuous on $\HH_+(\mathbb P^{m-1})$, because the norm topology on $\HH_+(\mathbb P^{m-1})$ is first countable and for maps whose domain is first countable topological spaces sequential continuity and continuity are equivalent.
To show sequential continuity of
$(\HH^{-1})|_{\HH_+(\PP^{m-1})}$ we take a sequence $(h_i)\subset \HH_+(\PP^{m-1})$ that converges to $h$.
Let $(\mu_i)\subset \mathcal{M}_+(\PP^{m-1})$ be the corresponding sequence so that $\HH(\mu_i)=h_i$ and
let $\mu$ be a measure with $\HH(\mu)=h$. We have to show that $\mu_i$ converges to $\mu$.
For this, we fix~$\phi\in C(\PP^{m-1})$ and show that $\langle \mu_i -\mu ,\phi \rangle \to 0$.
This would imply that $\mu_i -\mu\to 0$. Let $\eps>0$. By assertion (2) there are
$
\xi_1,\ldots,\xi_N\in\mathbb{P}^{m-1}
$
and
$c_1,\ldots,c_N\in\R$
such that
the function
$\psi (x) := \sum_{k=1}^N c_k  |\langle x, \xi_k\rangle |$ in $C(\PP^{m-1})$
satisfies $\|\psi-\phi\|_{\infty} < \eps/(2c)$.
We decompose
\begin{align}\label{decomposition}
\langle \mu_i -\mu,\phi \rangle
= \langle \mu_i , \phi - \psi \rangle
   + \langle \mu_i -\mu, \psi \rangle + \langle \mu, \psi - \phi\rangle.
\end{align}
The sequence of real numbers $\|h_i\|_{\infty}$ converges to $\|h\|_{\infty}$ and is thus bounded so that
there is $c>0$ such that $\sup_i \|h_i\|_{\infty} \le c$ and $\|h\|_{\infty}\le c$. An upper bound for the absolute value of third term in \cref{decomposition} is
$\vert\langle \mu,\phi - \psi\rangle\vert=\vert \int_{\mathbb P^{m-1}}(\phi(x) - \psi(x)) \;\mu(\mathrm{d} x)\vert \leq $
$\mu(\mathbb P^{m-1})\|\phi - \psi \|_{\infty}$ (here, we have used that $\mu$ is a measure and not a signed measure). Assertion (4) implies that this is bounded by $c\|\phi - \psi \|_{\infty}<\varepsilon/2$. We get the same bound for the first term. The middle term equals
$
\sum_{k=1}^N c_k   (h_i(\xi_k) -h(\xi_k))
$
and, by assumption, converges to zero for $i\to \infty$.
Therefore,
$
\limsup_i |\langle \mu_i -\mu,\phi \rangle |  \le \eps .
$
Since $\eps>0$ was arbitrary, we conclude that indeed
$\langle \mu_i -\mu,\phi \rangle \to 0$,
which proves assertion~(6).

Assertion~(6) follows from the noncontinuity of the tensor product of zonoids. More precisely, we will prove in \cref{thm:contoftens} below that the map $(K, L)\mapsto h_{K\otimes L}$ \emph{is not} sequentially continuous. On the other hand, we can write:
$h_{K\otimes L}=\HH(\tilde{T}(\HH^{-1}(h_K),\HH^{-1}(h_L))),$
where, for two measures $\mu_1, \mu_2$ the measure $\tilde{T}(\mu_1, \mu_2)$ is defined in \eqref{eq:Ttilde} below. By \cref{thm:contoftens} (3) below, the map $\tilde{T}$ is sequentially continuous and, if $\HH^{-1}$ were sequentially continuous, then $(K, L)\mapsto h_{K\otimes L}$ would also be sequentially continuous. This contradicts \cref{thm:contoftens}.
\end{proof}

By \cref{th:mainZ} (5), the cone $\HH_+(\PP^{m-1})$ coincides with the cone of support functions of zonoids and consequently $\HH(\PP^{m-1})$ coincides with the linear span of support functions of zonoids. Again by \cref{th:mainZ}, since the cone of nonnegative measures can be identified with the cone of zonoids, the vector space of measures can be identified with the vector space of virtual zonoids. In particular, on the space of virtual zonoids $\VZ(V)$ we can put two different topologies: the topology $\mathcal{T}_1$ induced by viewing them (through their support functions) as a subspace of continuous functions on $\mathbb P^{m-1}$  with the uniform convergence topology induced by the $\infty$-norm (we denote this topology by $\mathcal{T}_\infty$),
and the topology $\mathcal{T}_2$ induced by viewing them as signed measures with the weak--$*$ topology $\mathcal{T}_{\text{weak--$*$}}$. With this notation, letting $V\simeq \R^m$, we have
\be \label{eq:twotop}(\VZ(V), \mathcal{T}_1)\simeq (\HH(\PP^{m-1}), \mathcal{T}_{\infty})\quad \textrm{and} \quad (\VZ(V), \mathcal{T}_2)\simeq (\mathcal{M}(\PP^{m-1}), \mathcal{T}_{\text{weak--$*$}}).\ee
With these identifications, the map  $\HH$ can be seen as an inclusion:
\be \HH:(\VZ(V), \mathcal{T}_2)\stackrel{\mathrm{id}}{\longrightarrow} (\VZ(V), \mathcal{T}_1)\hookrightarrow (C(\mathbb P^{m-1}), \mathcal{T}_{\infty}).\ee
The fact that the inverse of $\HH$ \emph{is not} continuous means that the two topologies \eqref{eq:twotop} \emph{are not} the same.
What is remarkable, however, is that on the cone of zonoids, they induce the same topology.
\begin{remark}
Next to the weak--$*$ topology, another natural topology on $\mathcal M(\mathbb P^{m-1})$ is the one induced by the \emph{total variation norm} $\Vert \mu\Vert := \vert \mu\vert(\mathbb P^{m-1})$, where as in the proof of \cref{th:mainZ} we define $\vert \mu\vert :=\mu^++\mu^-$ for $\mu=\mu^+-\mu^-$ being the Hahn--Jordan decomposition of $\mu$. By \cref{remark_measures} below, for nonnegative measures this coincides with the length.
Let us observe that item (3) in \cref{th:mainZ} does not imply that this topology coincides with either $\mathcal{T}_1$ or~$\mathcal{T}_2$, not even when restricted to $\Z(V)$.
As an example, one can consider a sequence of pairwise different segments $[-v_i,v_i]$ such that all $v_i$ are in the sphere.
The corresponding measure is $\mu_i:=\tfrac{1}{2} \delta_{\Pi(v_i)}$, where the latter is the Dirac-delta measure. Observe, that for the total variation norm we have $\Vert \mu_i-\mu_j\Vert = \tfrac{1}{2}$. In particular such a sequence cannot converge in the topology induced by the total variation norm, but it can converge in $\mathcal{T}_1$ and $\mathcal{T}_2$.
\end{remark}

Next, we discuss the way of passing from the point of view of random vectors to the point of view of positive measures. Recall that for $x\in \mathbb R^m\setminus\{0\}$ and $f\in C(\PP^{m-1})$  we set $f(x):=(f\circ \Pi)(x)$, where $\Pi:\mathbb R^m\setminus\{0\}\to \mathbb P^{m-1}$ is the canonical projection.

\begin{proposition}\label{prop:fromrandtomeas}
Let $X\in\R^m$ be an integrable random vector with probability measure $\nu$. Let $\nu'$ be the measure such that $\nu'(A) = \int_A \Vert x\Vert \nu(\mathrm d x)$ for all measurable sets $A\subset \mathbb R^m$.
Then
$$h_{K(X)}=\HH(\mu),$$
where $\mu$ is the push-forward measure of $\nu'$ under the projection $\Pi: \mathbb R^m\setminus\{0\} \to \mathbb P^{m-1}$.
\end{proposition}
\begin{proof}
Let $u\in \mathbb P^{m-1}$. The support function of $K(X)$ evaluated at $u$ is $h_{K(X)}(u) = \tfrac{1}{2}\EE \vert \langle u,X\rangle \vert$. We write the integral explicitly as
\begin{equation}
\EE \vert \langle u,X\rangle \vert =  \int_{\mathbb R^m} \vert \langle u,x\rangle\vert \;\nu(\mathrm d x)
= \int_{\mathbb R^m} \vert \langle u,x/\Vert x\Vert \rangle \vert\;\nu'(\mathrm d x)
=\int_{\mathbb P^{m-1}} \vert \langle u,y \rangle \vert\;\mu(\mathrm d y),
\end{equation}
the second equality, because $\nu'$ is zero where $\Vert X\Vert =0$ and the third equality,
because $\vert \langle u,x/\Vert x\Vert \rangle $ is constant on preimages of $\Pi$.
\end{proof}

\begin{remark}\label{remark_measures} Let $X$ be an integrable vector, $K=K(X)$, and corresponding measure $\mu$.
\begin{enumerate}
\item From \cref{mydef:length} and \cref{prop:fromrandtomeas} it follows that $\ell(K)=2\mu(\PP^{m-1}) = \mu(S^{m-1})$.
\item If $X$ admits an even measurable density $\rho:\R^m\to \R$, then $\mu$ admits the density
    $
        \tilde{\rho}(x):=\int_{0}^{+\infty}t^m\rho(tx_0)\dd t,
   $
where $x_0\in S^{m-1}$ is such that $\Pi(x_0)=x$.
\end{enumerate}
\end{remark}

We close this section by proving that linear maps between spaces of measures are continuous with respect to the weak--$*$ topology.
\begin{lemma}\label{continuity_of_linear}
Let $M:V\to W$ be a linear map between Euclidean vector spaces, $m:=\dim V$ and $n:=\dim V$. Consider the induced linear map
$\widetilde{M}: \mathcal M(\mathbb P^{m-1}) \to \mathcal M(\mathbb P^{n-1})$ that sends the measure associated to the zonoid $K$ to the measure associated to the zonoid $\widehat{M}(K)$; that is, if $\HH(\mu) = \bar{h}_K$, then $\HH(\widetilde{M}(\mu)) = \bar{h}_{\widehat{M}(K)}$. Then, $\widetilde{M}$ is sequentially continuous with respect to the weak--$*$ topology.
\end{lemma}
\begin{proof}
Since $\widetilde{M}$ is linear, it is enough to show continuity at $0$. So, let $\mu_i$ be a sequence of measures converging to $0$, and let $\nu_i:=\widetilde{M}(\mu_i)$.
We have to show that $\nu_i$ converges to $0$. Let $K_i$ be the zonoid associated to $\mu_i$. Moreover, let us denote the pairing as in~\cref{def_inner_product}:  $\langle \mu, \phi\rangle=\int_{\mathbb P^{m-1}}\phi(x)\;\mu(\mathrm{d} x)$ for
$\phi \in C(\mathbb P^{m-1})$
From \cref{useful_properties_for_h}~(4) we know that, if $K$ is a zonoid, then $h_{\widehat{M}(K)}(v) = h_K(M^Tv)$. Take $v\in S(W)$. Then, we have:
$$\bar{h}_{\widehat{M}(K)}(v) = \begin{cases} \Vert M^Tv\Vert \cdot \bar{h}_K(M^Tv/\Vert M^Tv\Vert), & \text{ if } M^Tv \neq 0\\ 0, &\text{ else}\end{cases}.$$
This implies that for any $v$ such that $M^Tv\neq 0$ and for every $i$ we have:
$$\int_{\mathbb P^{n-1}}\vert\langle v, y\rangle\vert\; \nu_i(\mathrm{d} y) = \Vert M^T v\Vert \cdot  \int_{\mathbb P^{m-1}}\vert\langle M^Tv/\Vert M^Tv\Vert, x\rangle\vert\;\mu_i(\mathrm{d} x) =  \int_{\mathbb P^{m-1}}\vert\langle M^Tv, x\rangle\vert\;\mu_i(\mathrm{d} x),$$
and $\int_{\mathbb P^{n-1}}\vert\langle v, y\rangle\vert\; \nu_i(\mathrm{d} y) = 0$ otherwise. Since $(\mu_i)$ converges to $0$ we see that $\int_{\mathbb P^{n-1}}\vert\langle v, y\rangle\vert\; \nu_i(\mathrm{d} y)$ converges to $0$ for every $v\in S(W)$.
Now, we can proceed as we did in \cref{decomposition}, approximating any continuous function $\phi \in C(\mathbb P^{n-1})$ with a linear combination of functions of the form $y\mapsto \vert\langle v, y\rangle\vert$. This concludes the proof.
\end{proof}
\section{Tensor product of zonoids}\label{se:tensorproduct}

In this section we introduce and study the notion of tensor product of zonoids.
The only previous appearance of this notion we are aware of is~\cite[Definition 3.2]{AubrunLancien}.

In the whole section, $V$ and $W$ denote Euclidean spaces.
We start with the following central definition.

\begin{definition}[Tensor product of zonoids]\label{def:tensofzon2}
Let $K$ be a zonoid in $V$, $L$ be a zonoid in $W$
and $X\in V$ and $Y\in W$ be integrable random vectors
representing $K$ and $L$, respectively.
We define the \emph{tensor product} of $K$ and $L$ as
\be
 K\otimes L := K(X \otimes Y).
\ee
\end{definition}

Of course, we need to check that this tensor product does not depend
on the choice of random vectors representing the zonoids.
This is guaranteed by \cref{propo:tensorwd} below.
For stating it, we introduce the following notation:
for $x\in V$ and $y\in W$ we define
the following linear operators
\be \label{eq:defTy}
 T_{x} := \langle \cdot, x \rangle\otimes\mathrm{id}_W : V\otimes~W\to~W \quad\textrm{and}\quad
 T_{y}:=\mathrm{id}_V\otimes \langle \cdot, y \rangle: V\otimes~W\to~V.
\ee
Notice that their operator norms satisfy
\be \label{eq:nT}
 \|T_{x}\|_{\mathrm{op}}=\|x\|\quad \textrm{and}\quad \|T_{y}\|_{\mathrm{op}}=\|y\|.
\ee

\begin{lemma}\label{propo:tensorwd}
Let $K\in \ZZ(V)$ and $L\in \ZZ(W)$ be zonoids represented by
independent random vectors $X\in V$ and $Y\in W$, i.e., $K=K(X)$ and $L=K(Y)$. Then
\begin{equation}\label{eq:Tytens}
 h_{\ExpZon{X\otimes Y}}(u) = \EE_Y\, h_{\ExpZon{X}}(T_{Y}(u))=\EE_X h_{K(Y)}(T_{X}(u))
\end{equation}
and $K(X\otimes Y)\in \ZZo(V\otimes W)$ depends only on $K$ and $L$,
and not on the choice of the random vectors $X$ and $Y$.
\end{lemma}

\begin{proof}[Proof of \cref{propo:tensorwd}]
We first show that
\be
\forall u\in V\otimes W \quad \langle u, X\otimes Y \rangle = \langle T_{Y}(u), X \rangle.
\ee
It suffices to check this for simple tensors of the form
$u=v\otimes w$ where $v\in V$ and $w \in W$, because the simple vectors generate $V\otimes W$. For such $u$,
the definition of the scalar product in~$V\otimes W$ indeed implies
$$
 \langle v \otimes w, X\otimes Y \rangle =
 \langle v ,X \rangle \langle  w, Y \rangle =
 \langle \langle w,Y\rangle v,  X \rangle =
 \langle T_{Y}(v \otimes w), X \rangle .
$$
By \cref{eq:hE-form} the support function of the zonoid $\ExpZon{X\otimes Y}$ is given by
\begin{equation}\label{eq5}
h_{\ExpZon{X\otimes Y}}(u) =\tfrac{1}{2} \EE\, \left|\langle u, X\otimes Y \rangle\right|,
\end{equation}
where the expectation is over the joint distribution of $X$ and $Y$.
Taking first the expectation over $X$ and using the independence of $X$ and $Y$, we get
\begin{equation}
 h_{\ExpZon{X\otimes Y}}(u) = \EE_Y\, h_{\ExpZon{X}}(T_{Y}(u)).
\end{equation}
This shows that the dependence of
$\ExpZon{X\otimes Y}$ on $X$ is only through $K=K(X)$.
A symmetric argument for $Y$ completes the proof.
\end{proof}

\begin{example}[Tensor product of balls]
Let $X\in\R^k$ and $Y\in \R^m$ be independent, standard Gaussian vectors.
Then $K(\sqrt{2\pi} X)=B^k$ and $K(\sqrt{2\pi} Y)=B^m$
by~\cref{eq:ball}.
Therefore, 
\be
 B^k\otimes B^m=\sqrt{2\pi}K(X)\otimes \sqrt{2\pi}K(Y)=2\pi K(X\otimes Y).
\ee
The tensor product $K(X\otimes Y)$ was called \emph{Segre zonoid} in \cite{PSC}.
It is a convex body in the space $\R^k\otimes \R^m\simeq \R^{k\times m}$,
and its support function depends on singular values\footnote{In a recent
preprint \cite{SanyalSaunderson}, Sanyal and Saunderson have introduced the  notion of \emph{spectral convex bodies},
i.e., convex bodies in the space of symmetric operators whose support function depends on eigenvalues only.}
in the following sense.
Assuming $k\leq m$ and denoting by $\mathrm{sv}(M)\in \R^k$
the list of singular values of a matrix $M\in \R^{k\times m}$,
we have $
 h_{B^k\otimes B^m}(M) =(2\pi)^{\frac{1}{2}}g_k(\mathrm{sv}(M)),
$ by~\cite[Lemma 5.5]{PSC},
where the function $g_k:\R^k\to \R$ is defined by
$g_k(\sigma_1, \ldots, \sigma_k):=\EE\left(\sigma_1^2\xi_1^2+\cdots +\sigma_k^2\xi_k^2\right)^{\frac{1}{2}}$,
and~$\xi_1, \ldots, \xi_k$ are independent standard gaussians.
Recall from \cref{eq:rhodef} the definition of $\tau_m$. We obtain $\|B^k\otimes B^m\| = \tau_k/\sqrt{k}=:r_k$ from~\cite[Lemma 5.5]{PSC},
which remarkably only depends on~$k$. Thus $r_k B^{km}$ is the smallest centered ball containing $B^k\otimes B^m$.
In~\cite[Theorem~6.3]{PSC}, it was shown that for fixed $k$ and $m\to\infty$,
$B^k\otimes B^m$ is not much smaller in volume than $r_k B^{km}$: we have
$\log \vol_{km}(B^k\otimes B^m) = \log \vol_{km}(r_k B^{km}) - O(\log m)$ for $m\to\infty$.
\end{example}

The next result shows that the tensor product of zonoids behaves well with respect to Minkowski addition,
scalar multiplication, norm, and inclusion.

\begin{proposition}\label{prop:contOfTens} 
The tensor product of zonoids is componentwise
Minkowski additive and positively homogenous.
Moreover, the tensor product is monotonically increasing in each variable; that is,
$K_1\subset K_2$ and $L_1\subset L_2$ implies
$K_1\otimes L_1 \subset K_2\otimes L_2$.
Finally, the tensor product of zonoids is associative.
\end{proposition}

\begin{proof}
Given $K_1, K_2\in \ZZ(V)$, $\lambda_1, \lambda_2 \ge 0$, and a random vector $Y\in W$
representing~$L$, we use~\cref{eq:Tytens}
to write the support function of $(\lambda_1K_1+\lambda_2K_2)\otimes L$ as
\begin{align}h_{(\lambda_1K_1+\lambda_2K_2) \otimes L}(u)& = \EE_Y\, h_{\lambda_1K_1+\lambda_2K_2}(T_{Y}(u))\\
&=\EE_Y\,\lambda_1 h_{K_1}(T_{Y}(u))+\EE_Y\, \lambda_2 h_{K_2}(T_{Y}(u))\\
&=\lambda_1h_{K_1\otimes L}(u)+\lambda_2 h_{K_2\otimes L}(u).
\end{align}
For the second factor we argue analogously.
Therefore Minkowski additivity and positive homogeneity in each factor follows from
\cref{useful_properties_for_h}.

For the monotonicity we assume in addition that $K_1\subset K_2$.
\cref{useful_properties_for_h} implies $h_{K_1}\leq h_{K_2}$.
Again, \cref{eq:Tytens} gives
$h_{K_1\otimes L} (u)  =\EE_Y\, h_{K_1}(T_{Y}(u))$ and $h_{K_2\otimes L} (u)  =\EE_Y\, h_{K_2}(T_{Y}(u))$.
Therefore, 
$h_{K_1\otimes L}\leq h_{K_2\otimes L}$.
Again using \cref{useful_properties_for_h} shows $K_1\otimes L\subset K_2\otimes L$.
For the second factor we argue analogously.
Finally, the associativity immediately follows from the one of the usual tensor product.
\end{proof}

\begin{example}[Tensor product of zonotopes]\label{ex:tpz}
The tensor product of symmetric segments is given by
\be\label{eq:zonoT}
 \tfrac{1}{2}[-v_1, v_1]\otimes \cdots\otimes \tfrac{1}{2}[-v_p, v_p]
 =\tfrac{1}{2}[-v_1\otimes \cdots \otimes v_p, v_1\otimes \cdots \otimes v_p] ,
\ee
where $v_1\in V_1, \ldots, v_p\in V_p$.
(Indeed, just take for $X_j\in V_j$ a random variable taking the constant value $v_j$
and note that $K(X_j)=\tfrac{1}{2}[-v_j,v_j]$.)
Together with the biadditivity of the tensor product (\cref{prop:contOfTens}),
we conclude that the tensor product of two zonotopes is
\begin{equation}\label{tensor_product_segments}
   \Big(  \sum_{i=1}^n\tfrac{1}{2}[-x_i,x_i]\Big)\otimes
   \Big(  \sum_{j=1}^m\tfrac{1}{2}[-y_j,y_j]\Big)
    =\sum_{i=1}^n\sum_{j=1}^m\tfrac{1}{2} [-x_i\otimes y_j,x_i\otimes y_j].
\end{equation}
\end{example}

We now show that the length is multiplicative with respect to the tensor product and
prove an upper bound on the norm of a tensor product of zonoids.

\begin{proposition}\label{prop:normTP}
For $K\in\K(V)$ and $L\in\K(W)$ such that $m=\dim V \le \dim W$, we have
$$
  \ell(K \otimes L) = \ell(K) \ell(L) , \quad
  \| K \otimes L \| \ \le \ 2\sqrt{m} \, \| K \|  \| L \| .
$$
\end{proposition}

\begin{proof}
Suppose that $K=K(X)$ and $L=K(Y)$
with independent random vector $X\in V$ and $Y\in W$.
Then, by \cref{mydef:length},
$
 \ell(K \otimes L) = \EE \|X \otimes Y \| = \EE \|X\| \cdot \EE \|Y\| =  \ell(K) \ell(L) ,
$
showing the first assertion.

For the norm inequality, assume $V=\R^m$ and $W=\R^n$ and w.l.o.g.\ $m\le n$.
Recall that the nuclear norm of a matrix $M\in\R^{m\times}$
is defined as the sum of its singular values.
The corresponding unit ball $B_\mathrm{nuc}$
equals the convex hull of
the rank one matrices $v\otimes w$ such that
$v\in\R^m$ and $w\in \R^n$ have norm one, e.g., see~\cite{derksen:16}.
If we denote by $B$ the unit ball with respect to the Frobenius norm,
we get $B \subset \sqrt{m} B_\mathrm{nuc}$,
where we used that $m\le n$.
We obtain
$$
 \|K\otimes L\| = \max_{u\in B} h_{K\otimes L}(u)
 \ \le \
 \sqrt{m} \max_{u\in B_\mathrm{nuc}} h_{K\otimes L}(u)
  = \tfrac 12 \sqrt{m} \max_u \EE | \langle u, X\otimes Y\rangle | .
$$
But for $u=v\otimes w$ with unit vectors $v,w$, we have
$$
 \EE  |\langle v\otimes w, X\otimes Y\rangle | =
 \EE  |\langle v, X\rangle | \cdot \EE |\langle w, Y)\rangle | =
  4 h_K(v) h_L(w)  \ \le \
  4 \|K\| \cdot \|L\| .
$$
Using the convexity of $h_{K\otimes L}$, implies the second assertion.
\end{proof}

It is straightforward to extend the tensor product of zonoids to a bilinear map between spaces of virtual zonoids.

\begin{proposition}[Tensor product of virtual zonoids]\label{propo:tp}
The tensor product of zonoids from \cref{def:tensofzon2} uniquely extends to a bilinear map
$\widehat{T}: \VZ(V)\times \VZ(W)\to \VZ(V\otimes W)$.
The resulting tensor product of virtual zonoids is associative.
\end{proposition}

\begin{proof}
The only possible way to define the map $\widehat{T}$ is by setting
\be
 (K_1 - K_2)\otimes (L_1 - L_2)
 := K_1\otimes L_1+K_2\otimes L_2-K_1\otimes L_2 - K_2\otimes L_1.
\ee
Using the biadditivity of the tensor product of zonoids (\cref{prop:contOfTens})
it is straightforward to check that this is well defined and defines a bilinear map.
The associativity follows from the associativity of the tensor product of zonoids.
\end{proof}

\subsection{Continuity of the tensor product} \label{sec:continuity}

Here we discuss the continuity of the tensor product map.
The main result is that the tensor product is continuous on zonoids,
but not on virtual zonoids with the norm topology. It is only
separately continuous in each variable, meaning that it is continuous in each component.
However, viewing virtual zonoids as measures via the correspondence described
in~\cref{sec:VZandmeas} and endowed with the weak--$*$ topology,
the tensor product turns out to be sequentially continuous; see~\cref{thm:contoftens} below.

As above we denote by
\begin{equation}
\widehat{T}:\VZ(V)\times\VZ(W)\to\VZ(V\otimes W)
\end{equation}
the tensor product map, $\widehat{T}(K,L)=K\otimes L$, and its restriction to zonoids is denoted
\begin{equation}
T:\Z(V)\times\Z(W)\to\Z(V\otimes W).
\end{equation}
Viewing virtual zonoids as measures,
we obtain a bilinear map
\begin{equation}
\mT:\cM(\PP(V)) \times \cM(\PP(W))\to\cM(\PP(V\otimes W)).
\end{equation}
Concretely, this map can be described as follows.
Let $\mu \in \cM(\PP(V))$ and $\nu\in\cM(\PP(W))$,
and let $K_\mu\in\VZ(V)$, $\K_\nu\in\VZ(W)$
be such that $h_{K_\mu}=\HH(\mu)$ and $h_{K_\nu}=\HH(\nu)$,
where $\HH$ is the cosine transform from \cref{def_cosine_transform}.
Then $\mT(\mu,\nu)$ is the measure on $\PP(V\otimes W)$
characterized by~$\HH\big( \mT(\mu,\nu)\big)=h_{K_\mu \otimes K_\nu}$.

We now show that the map $\mT$ has a direct
natural characterization.
For this, we first recall that the Segre embedding
$
 \PP(V)\times \PP(W) \to \PP(V\otimes W),\,
  [v] \otimes [w] \mapsto [v\otimes w] ,
$
is an isomorphism onto its image, which allows
to view $\PP(V)\times\PP(W)$ as a subspace of
$\PP(V\otimes W)$. Taking the pushforward gives
\be\label{eq:Ttilde}
\mT: \cM(\PP(V))\times  \cM(\PP(W)) \to \cM(\PP(V\otimes W)),
\ee
which is a bilinear map.

\begin{lemma}\label{lem:tensofmeas}
    The map $\mT:\cM(\PP(V))\times\cM(\PP(W))\to\cM(\PP(V\otimes W))$
    equals
    the tensor product of measures composed with the pushforward of the Segre map.
\end{lemma}

\begin{proof}
Since the map $\mT$ is bilinear we can assume without loss of generality that
$\mu\in \cM(\PP(V))$ and $\nu \in \cM(\PP(W))$ are probability measures;
the general case then follows by homogeneity and linearity.
In that case, by \cref{prop:fromrandtomeas}, $K_\mu=K(X)$ where $X\in S(V)$ is a random vector of law $\mu$
(recall that we identify a measure on $\PP(V)$ with the corresponding even measure on the sphere). Similarly, $K_\nu=K(Y)$, where $Y\in S(W)$ is a random vector of law $\nu$,
that we can assume to be independent of $X$.
By definition, we have $K_\mu \otimes K_\nu=K(X\otimes Y)$.
The law of $X\otimes Y$ is the pushforward by the Segre map
of the tensor product of measures
$\mu\otimes\nu\in\cM_+(\PP(V)\times\PP(W))$,
and this concludes the proof.
\end{proof}

One could take \cref{lem:tensofmeas} as the definition of the tensor product of zonoids,
as it may appear simpler.
This simplicity however entirely relies on the fact that the tensor product on vectors
(the Segre map) sends the product of spheres to the sphere. When later in \cref{sec:FTZC},
we will deal with multilinear maps that do not have this property, the point of view of
random vectors is easier to handle.

To investigate the continuity of $T$ and $\widehat{T}$, let us prove an inequality.

\begin{lemma}\label{le:Lipsch}
For $K_1,K_2 \in \Z(V)$ and $L \in \Z(W)$ we have with $m:=\dim W$
$$
 \|K_1\otimes L - K_2\otimes L\| \ \le\
 \tau_m \|L\| \|K_1 - K_2 \|,
$$
where $\tau_m$ is defined as in \cref{eq:rhodef}.
\end{lemma}
\begin{proof}
Let $u\in V\otimes W$ be such that $\|K_1\otimes L - K_2\otimes L\| = \vert h_{K_1\otimes L}(u) - h_{K_2\otimes L}(u)\vert$. Let $L=K(Y)$.
From~\cref{eq:Tytens} we get
$$
 h_{K_1\otimes L}(u) - h_{K_2\otimes L}(u) =
 \EE_Y\, \big(h_{K_1} - h_{K_2}\big) ((T_{Y}(u)) ,
$$
hence,
\begin{align}
 |h_{K_1\otimes L}(u) - h_{K_2\otimes L}(u)|
 &\ \le\
  \EE_Y\, |\big(h_{K_1} - h_{K_2}\big) ((T_{Y}(u))| \\
   &\ \le\
  \EE_Y\, \|h_{K_1} - h_{K_2}\|_\infty \|T_{Y}(u))\| \\
  &\ \le\
  \EE_Y\, \|Y\| \cdot \|K_1 - K_2\|\ =\ \ell(L)\cdot \|K_1 - K_2\|
\end{align}
where we used~\cref{eq:nT} for the third inequality.
Applying \cref{cor:ideno} completes the proof.
\end{proof}

The main continuity properties are summarized in the following result.

\begin{theorem}\label{thm:contoftens}
Suppose that  $\dim V, \dim W \ge 2$. Then, the tensor product map satisfies the following.
\begin{enumerate}
\item $T: \ZZ(V)\times \ZZ(W)\to \ZZ(V\otimes W)$ is continuous.
More specifically,
for $K_1,K_2 \in \Z(V)$ and $L_1,L_2 \in \Z(W)$, we have
\begin{equation}
 d_H\left(K_1\otimes L_1, K_2\otimes L_2\right) \ \le\
  \left(\tau_m\|L_1\| + \tau_n\|K_2)\| \right)
  \left( d_H\left(K_1 ,K_2\right) +  d_H\left(L_1 , L_2\right) \right),
\end{equation}
where $n:=\dim V$ and $m:=\dim W$;

\item $\widehat{T}: \VZ(V)\times \VZ(W)\to \VZ(V\otimes W)$
with the norm topology on both sides is not sequentially continuous, but
separately (i.e., componentwise) continuous;

\item $\mT:\cM(\PP(V))\times\cM(\PP(W))\to\cM(\PP(V\otimes W))$ with the weak--$*$ topology
on both sides is sequentially continuous.
\end{enumerate}
\end{theorem}

\begin{proof}
For proving $(1)$, recall that $d_H(K,L)=\|K-L\|$.
From the multiadditivity of the tensor product and the triangle inequality of the norm,
we get
$$
\|K_1\otimes L_1 - K_2\otimes L_2\|
 \ \le\ \|K_1\otimes L_1 - K_2\otimes L_1\| + \|K_2\otimes L_1 - K_2\otimes L_2\| .
$$
Combined with \cref{le:Lipsch}, this yields
$$
\|K_1\otimes L_1 - K_2\otimes L_2\| \ \le\
 \tau_m \|L_1\| \|K_1 - K_2 \| +
\tau_n \|K_2\| \|L_1 - L_2 \| ,
$$
which proves the first assertion.

As for $(2)$, the separate continuity follows directly from \cref{le:Lipsch}.
To prove that $\widehat{T}$  is not (sequential) continuous, we begin with a general observation.
Let $\varphi\colon E \times F \to G$ be a bilinear map of
real normed vector spaces. Then $\varphi$ is (sequential) continuous if and only if
it has finite operator norm:
$$
 \|\varphi\|_{\mathrm{op}} := \sup_{\|x\|\le 1, \|y\|\le 1} \|\varphi(x,y)\| \ <\  \infty .
$$
We show now that $\widehat{T}$ has infinite operator norm.
It suffices to prove this for $V=W=\R^2$.
Consider the sequence of vectors $a_n:=(n, 1)$, $b_n:=(n,0)$
and the corresponding sequence of segments $A_n:=\frac{1}{2}[-a_n, a_n]$, $B_n:=\frac{1}{2}[-b_n, b_n]$
in $\R^2$. This defines the sequence of virtual zonoids $A_n- B_n\in \VZ(\R^2)$.
It is immediate to check that
$$
 \|A_n-B_n\| = d_H(A_n,B_n) = \tfrac12 .
$$
Consider $P_n := (A_n-B_n)\otimes (A_n-B_n) \in \VZ(\R^2\otimes\R^2)$.
It suffices to show that
$\lim_{n\to \infty} \|P_n \| = \infty$.
For this, we calculate
\be
a_n\otimes a_n=\begin{bmatrix} n^2 & n \\n & 1\end{bmatrix},\
b_n\otimes b_n=\begin{bmatrix} n^2 & 0 \\0 & 0\end{bmatrix},\
a_n\otimes b_n=\begin{bmatrix} n^2 & 0 \\n & 0\end{bmatrix},\
b_n\otimes a_n=\begin{bmatrix} n^2 & n \\0 & 0\end{bmatrix}.
\ee
Their inner product with the matrix
$w_n := \begin{bmatrix} 1 & -n \\ -n &  0\end{bmatrix}$
is given by
\be
\langle a_n\otimes a_n, w_n\rangle = -n^2,\
\langle b_n\otimes b_n, w_n\rangle = n^2,\
\langle a_n\otimes b_n, \
w_n\rangle= \langle b_n\otimes a_n, w_n\rangle=0,
\ee
Using~\cref{eq:hE-form}, we obtain
$h_{A_n\otimes A_n}(w_n) = \tfrac12 | \langle a_n \otimes a_n , w_n \rangle | = \frac{n^2}{2}$,
and similarly
$h_{B_n\otimes B_n}(w_n) = \frac{n^2}{2}$, and
$h_{A_n\otimes B_n}(w_n) = h_{B_n\otimes A_n}(w_n) = 0$.
Therefore,
$$
 \frac{h_{P_n}(w_n)}{\|w_n\|} = \frac{n^2}{\|w_n\|} = \Omega(n),
$$
which completes the proof of the second item.

For item (3) we recall from \cref{lem:tensofmeas} that $\mT$
equals
the tensor product of measures composed with the pushforward of the Segre map. The pushforward of a measure under a continuous map is weak--$*$ continuous. Mapping two measures to their product measure is sequentially continuous by \cite[Theorem 2.8]{Billingsley}. This finishes the proof for the third assertion.
\end{proof}

By \cref{propo:banach}, the completion of the normed vector space $\VZ(V)$ is isomorphic to
the Banach space of even real valued continous function defined on the unit sphere of $V$,
or equivalently to $C(\PP(V))$.
One may hope to extend the tensor product map
to the completions to obtain a separate continuous bilinear map.
We show now that, unfortunately, this is impossible.
For proving this, we will rely on \cref{thm:FTZCdet} below and on \cite[Theorem~5.2.2]{bible}.

\begin{proposition}\label{propo:notpossible}
There is no bilinear map
$C(\PP(V))\times C(\PP(W))\to C(\PP(V\otimes W))$ that is
separate continuous and extends
the tensor product map $\widehat{T}\colon \VZ(V)\times \VZ(W)\to \VZ(V\otimes W)$,
provided $\dim V, \dim W \ge 2$.
\end{proposition}

\begin{proof}

Suppose by way of contradiction that there is a bilinear map as in the proposition.
From this it is straightforward to construct such map
$\varphi\colon C(\PP^1)\times C(\PP^1)\to C(\PP^3)$
for $V=W=\R^2$. Thus it is enough to prove it for this case.
The determinant map $\R^2 \times \R^2\to \R$ is a bilinear map and hence it factors
via a linear map
$M\colon \R^2\otimes \R^2\to \R$.
By \cref{thm:FTZCdet} below, the associated linear map
$\widehat{M}\colon\VZ(\R^2\otimes \R^2)\to \VZ(\R) \simeq\R$ gives
the mixed volume of the tensor product in the sense that,
for $K_1,K_2\in\Z(\R^2)$,
$$
 \widehat{M}(K_1\otimes K_2) = \ell (K_1 \wedge K_2) = 2 \,\MV(K_1, K_2) .
$$
Since $\widehat{M}$ is continuous (\cref{propo:linearext}),
we can extend it to a continuous linear map
$\lambda\colon  C(\PP(\R^2\otimes \R^2)) \to C(\PP^0) \simeq \R$
by Hahn-Banach. Then,
$\lambda\circ\varphi\colon C(\PP^1)\times C(\PP^1) \to \R$
is a componentwise continuous bilinear map extending
$2\cdot\widehat{\det}\colon\Z(\R^2) \times \Z(\R^2)\to \R$.
Now recall that $C(\PP^1)$ is identified with the even functions in $C(S^1)$.

The existence of $\lambda\circ\varphi$ contradicts the fact that there is no
separate continuous bilinear function
$V\colon C(S^1) \times C(S^1) \to \R$ that
satisfies
$V(\bar{h}_K,\bar{h}_L) = \MV (K,L)$ for $K,L\in \Z(\R^2)$.
The latter follows by inspecting the proof of \cite[Theorem 5.2.2]{bible},
where the analogous statement is made about bilinear functions $V$
satisfying the above condition for all $K,L\in \K(\R^2)$.
\end{proof}

\section{Multilinear maps induced on zonoids}\label{sec:FTZC}

In this section we show how to associate with any multilinear map between Euclidean spaces
a corresponding multilinear map of the corresponding vector spaces of zonoids.
This will allow us to construct the zonoid algebra.
The zonoid algebra inherits a duality notion from the Hodge star operator.

\subsection{Zonoids and multilinear maps}

We learned how to associate with a linear map $M\colon V\to W$ of Euclidean spaces
a continuous, order preserving linear map
$\widehat{M}\colon \VZ(V)\to \VZ(W)$ of the corresponding vector spaces of zonoids (\cref{def:linearext}).
We now transfer this construction to multilinear maps via the tensor product.
Note that if $V_1, \ldots, V_p$ are Euclidean vector spaces, then each
$\VZ(V_i)$ is an partially ordered and normed real vector space. The product space
$\VZ(V_1)\times \cdots \times \VZ(V_p)$ carries a componentwise order
and the product topology, where $\VZ(V_i)$ carries the topology defined
by the norm of virtual zonoids from \cref{th:virt:zonoid}.

\begin{theorem}[Induced multilinear zonoid maps]\label{thm:FTZC}
Let $V_1, \ldots, V_p$ and $W$ be Euclidean vector spaces and
$M:V_1\times \cdots \times V_p\to W$ be a multilinear map.
There exists a unique multilinear, separate continuous map
\be
\widehat{M}:\VZ(V_1)\times \cdots \times \VZ(V_p)\to \VZ(W) ,
\ee
such that for every $v_1\in V_1, \ldots, v_p\in V_p$
$$
\widehat{M}\left(\tfrac{1}{2}[-v_1, v_1], \ldots, \tfrac{1}{2}[-v_p, v_p]\right)
 =\tfrac{1}{2}\left[-M(v_1, \ldots, v_p), M(v_1, \ldots, v_p)\right] .
$$
Restricting to zonoids, we get a continuous map
$\ZZo(V_1)\times \cdots\times  \ZZo(V_p) \to \ZZo(W)$.
The map $\widehat{M}$ preserves the componentwise inclusion order of zonoids. If we interpret this map on the level of signed measures, then we obtain a multlinear map $\widetilde{M}$, which is sequentially continuous with respect to the weak--$*$ topology.
\end{theorem}
\begin{proof}
To show existence, we rely on
the universal property of tensor product:
there is a unique linear map $L:V_1\otimes\cdots\otimes~V_p\to~W$
such that $L(v_1\otimes \cdots \otimes v_p) = M(v_1,\ldots,v_p)$.
Consider the linear continuous map
$\widehat{L}:\VZ(V_1\otimes \cdots \otimes V_p)\to \VZ(W)$
given by \cref{def:linearext}.
For $(K_1, \ldots, K_p)\in \VZ(V_1)\times \cdots \times \VZ(V_p)$,
we define the map $\widehat{M}$ by
\be\label{eq:md}
\widehat{M}(K_1,\ldots,K_p):= \widehat{L}(K_1\otimes\cdots\otimes K_p).
\ee
This is the composition of the linear map $\widehat{L}$ with the multilinear tensor product map
from \cref{propo:tp}, therefore it is multilinear.

Restricting to zonoids and using \cref{prop:contOfTens},
we see
$\widehat{M}(\ZZo(V_1)\times \cdots\times  \ZZo(V_p))\subseteq \ZZo(W)$.
The asserted formula for the image of $\widehat{M}$ on tuples of segments
is a direct consequence of~\cref{eq:zonoT} and the definition of the map $\widehat{M}$.

Since $\widehat{L}$ is continuous and the tensor product map from \cref{propo:tp}
is separate continuous, $\widehat{M}$ is separate continuous.
Similarly, $\widehat{M}|_{\ZZo(V_1)\times \cdots\times  \ZZo(V_p)}$
is continuous since the tensor product map on zonoids is continuous (\cref{prop:contOfTens}).

For the uniqueness of the map $\widehat{M}$ we argue as follows.
Since by definition, any zonoid in $V_i$ can be approximated by symmetric segments and
the values of $\widehat{M}$ are determined on tuples of segments,
the componentwise continuity of $\widehat{M}$ determines $\widehat{M}$
on $\Z(V_1)\times \cdots\times \Z(V_p)$. In turn, this determines $\widehat{M}$ by multilinearity.

By \cref{prop:contOfTens}, the tensor product of zonoids preserves the componentwise order.
Moreover, $\widehat{L}$ preserves the order by \cref{propo:linearext}.
This implies that $\widehat{M}$ preserves the componentwise order.

The last statement about the sequential continuity with respect to the weak--$*$ topology follows from \cref{continuity_of_linear} and \cref{thm:contoftens} (3).
\end{proof}

\begin{remark}
The map
$\widehat{M}:\VZ(V_1)\times \cdots \times \VZ(V_p)\to \VZ(W)$
is in general not sequentially continuous for the norm topology. Indeed \cref{thm:contoftens} shows
that the tensor product map for two factors is not (sequentially) continuous, and this
immediately extends to any number of factors. By contrast, we showed that $\widetilde{M}$ is sequentially continuous for the weak--$*$ topology.
\end{remark}

Our construction is nicely compatible with the description of zonoids by random vectors.

\begin{corollary}\label{propo:MK}
Let $X_1\in V_1, \ldots, X_p\in V_p$ be integrable and
independent random vectors. For a multilinear map
$M:V_1\times \cdots \times V_p\to W$ we have
\be
\widehat{M}(K(X_1),\ldots, K(X_p))=K(M(X_1, \ldots, X_p)).\ee
\end{corollary}

\begin{proof}
This follows from the construction \cref{eq:md} of the map $\widehat{M}$ and \cref{lemma:linear}.
\end{proof}

A recent work that can be interpreted from the point of view of multilinear functions of zonoids is the paper \cite{FCB} by Meroni and Mathis, in which they study so-called fiber bodies.

\subsection{Zonoid algebra}\label{section:ZonoidAlgebra}

We assign to the exterior powers of a Euclidean space
the corresponding zonoid vector spaces and apply \cref{thm:FTZC} to the wedge product
of the exterior algebra to arrive at a commutative and associative graded algebra,
which we call the {\em zonoid algebra} of~$V$.

Let $V$ denote a Euclidean vector space of dimension $m$ throughout this section.
Consider the $d$--th exterior power $\Lambda^d V$ of $V$,
defined for $d\in\N$, and note that $\Lambda^d V = 0$ if $d > m$,
This space inherits a Euclidean structure from~$V$,
which can be described as follows:
if $\{e_1, \ldots, e_m\}$ is an orthonormal basis for $V$,
an orthonormal basis for $\Lambda^d V$ is given by
$\{e_{i_1}\wedge \cdots \wedge e_{i_d}\}_{1\leq i_1<\cdots <i_d\leq m}$.
We form the direct sum of zonoid vector spaces
\be\label{eq:AV}
 \A(V) := \bigoplus_{d=0}^{m}\VZ(\Lambda^d V).
\ee
The elements of $\VZ(\Lambda^kV)$ will be said to have degree $d$.
By \cref{re:1-dim} we shall identify 
$\VZ(\Lambda^0 V)= \VZ(\R) \simeq\R$.

\begin{remark}
\cref{propo:banach} implies that $\A(V)$ is an infinite dimensional real vector space,
unless $\dim(V)\le 1$.
\end{remark}

Via \cref{thm:FTZC}, we associate with each wedge product
$\wedge\colon\Lambda^d V\times \Lambda^e V\to \Lambda^{d+e} V$
a componentwise continuous bilinear map
\begin{equation}\label{eq:zonoidwedge} 
\timeszon : \VZ\big(\Lambda^d V\big)\times \VZ\big(\Lambda^e V\big)\to \VZ\big(\Lambda^{d+e} V\big),
\end{equation}
and, by extension, a componentwise continuous bilinear map
$\A(V) \times \A(V) \to \A(V)$,
all denoted by the same symbol.
We call this map the \emph{wedge product of virtual zonoids}.
\cref{thm:FTZC} implies that the wedge product of zonoids is a zonoid. We write $K^{\wedge d}$ for the wedge of $K$ with itself $d$ many times.

In an analogous way, we define the wedge product of several (virtual) zonoids.
We can describe this product explicitly as follows.
Assume that $X_1\in \Lambda^{d_1}V,\ldots,X_p\in \Lambda^{d_p}V$ are independent random vectors
representing the zonoids $A_j\in \ZZo(\Lambda^{d_j}V)$.
Then \cref{propo:MK} implies that
\be\label{eq:www}
 A_1\wedge\cdots \wedge A_p = K(X_1)\wedge \cdots \wedge K(X_p)=K(X_1\wedge \cdots \wedge X_p) \in \ZZo(\Lambda^{d_1+\cdots+d_p}V).
\ee

We call $\A(V)$ the \emph{zonoid algebra} associated with $V$.
This naming is justified by the following theorem.

\begin{theorem}\label{th:zonoidalgebra}
The wedge product turns $\A(V)$ into a graded, associative and commutative real algebra.
The wedge maps of zonoids
\be
 \ZZ(\Lambda^{d_1}V)\times \cdots\times \ZZ(\Lambda^{d_p}V)\to \ZZ(\Lambda^{d_1+\cdots + d_p}V),\:
  (A_1, \ldots, A_p)\mapsto A_1\wedge \cdots \wedge A_p
 \ee
are continuous.
These maps preserve the inclusion order of zonoids:
if we have $A_j'\subset A_j$ for zonoids in $\ZZ(\Lambda^{d_j}V)$, then
\be\label{wedge_inclusion}
A_1'\wedge \cdots \wedge A_p' \ \subset\ A_1\wedge \cdots \wedge A_p.
\ee
Moreover, the wedge product of zonoids does not increase the length:
$$
\ell(A_1\wedge \cdots \wedge A_p) \ \le\ \ell(A_1) \cdots \ell(A_p) .
$$
\end{theorem}

\begin{proof}
The associativity follows from the associativity of the wedge product and \cref{eq:www}.
The distributivity is a consequence of \cref{propo:tp}.
The gradedness follows from the definition of $\A(V)$. The multiplicative unit lies in $\VZ(\Lambda^0 V)= \VZ(\R) \simeq\R$.
The commutativity of the wedge follows with \cref{eq:www}
from the known relation $X\wedge Y= \pm Y\wedge X$ and the fact that $K(-Z)=K(Z)$.

The wedge map of zonoids is continuous since
it is obtained by composing a continuous linear map with the continuous tensor product of zonoids
(see \cref{thm:contoftens}).
The preservation of the inclusion order follows from \cref{thm:FTZC}.

For the length inequality, we use that the antisymmetrization map
$\otimes_j \Lambda^{d_j} V \to \Lambda^{d_1 +\dots + d_p} V$
is an orthogonal projection. Hence
$\ell(A_1\wedge \cdots \wedge A_p) \ \le\ \ell(A_1\otimes \cdots \otimes A_p)= \ell(A_1)\cdots \ell(A_p)$
by \cref{propo:linearext} and \cref{prop:normTP}.
\end{proof}

Here is an immediate yet important observation about wedge products of zonoids.
\begin{lemma}\label{wedge_is_zero}
Let $K\in \Z(V)$ be a zonoid. Recall that we defined the dimension of $K$ as the dimension of its linear span $\spanK{K}$. Then, $K^{\wedge d}=0$ for all $d>\dim(K)$.
\end{lemma}
\begin{proof}
Let us write $K=K(X)$. By \cref{eq:www} we have $K^{\wedge d} = K(X_1\wedge \cdots\wedge X_d)$, where $X_1,\ldots,X_d$ are independent copies of $X$. With probability one we have $X\in \spanK{K}$, and so with probability one the $X_i$ are linearly dependent. Hence, $X_1\wedge \cdots\wedge X_d=0$ almost surely, so that $K^{\wedge d}=0$.
\end{proof}

In the next section, we will link the length to the mixed and intrinsic volumes.
More specificially, \cref{th:IVL} show that the $j$th intrinsic volume
$V_j(K)$ can be expressed as $\frac{1}{j!}\ell(K^{\wedge j})$.
This allows to immediately derive \cref{hug_result}, which is a reverse Alexandroy-Fenchel
inequality, that was independently found very recently by
B\"or\"oczky and Hug~\cite{B-Hug:21}.

It is useful to denote by $K[d]$ the zonoid $K$ repeated $d$ times.

\begin{corollary}\label{hug_result}
Let $K_1,\ldots,K_p\in\Z(V)$ be zonoids and $\spanK{K_1},\ldots,\spanK{K_p}$ their spans.
Assume $d_1,\ldots,d_p\in \N$ satisfy $d_1+\ldots+d_p=m$. Then
\be
 \frac{m!}{d_1!\cdots d_p!}\, \MV(K_1[d_1],\ldots,K_p[d_p]) \ \le\ V_{d_1}(K_1)\cdots V_{d_p}(K_p) .
\ee
Equality holds if and only
$\spanK{K_1}\ldots,\spanK{K_p}$ are pairwise orthogonal or if $K_i^{\wedge d_i}=0$ for at least one $i$.
\end{corollary}
By \cref{wedge_is_zero} the condition that $K^{\wedge d}=0$ is equivalent to either $K=0$ or $d>\dim \spanK{K}$.
\begin{proof}[Proof of \cref{hug_result}]
By \cref{th:IVL}, the stated inequality can be rephrased as
\be\label{eq:LeIn}
  \ell\big(K_1^{\wedge d_1}\wedge \cdots\wedge K_p^{\wedge d_p}\big) \ \le\
  \ell\big(K_1^{\wedge d_1}\big)\cdots \ell\big(K_d^{\wedge d_p}\big) ,
\ee
which is a consequence of the submultiplicativity of the length,
see \cref{th:zonoidalgebra}.

For analyzing when equality holds, we assume $p=2$ to simplify notation.
We write $K_1=K(X)$ and $K_2 = K(Y)$. Let $X_1,\ldots,X_{d_1}$ be independent copies of $X$ and $Y_1,\ldots,Y_{d_2}$ be independent copies of $Y$.
The above inequality~\cref{eq:LeIn} can be written as an inequality of expectations
$
 \EE \|X_1\wedge \cdots \wedge X_{d_1}\wedge Y_1\wedge \cdots \wedge Y_{d_2}\| \ \le\
 \EE \|X_1\wedge \cdots \wedge X_{d_1}\| \cdot \EE\|Y_1\wedge \cdots \wedge Y_{d_2}\| ,
$
and equality holds if and only if   $\|X_1\wedge \cdots \wedge X_{d_1}\wedge Y_1\wedge \cdots \wedge Y_{d_2}\| =
\|X_1\wedge \cdots \wedge X_{d_1}\| \cdot \|Y_1\wedge \cdots \wedge Y_{d_2}\|$
almost surely.
We have equality if and only if
$\langle X_i,Y_j\rangle = 0$ almost surely for all $i,j$, or if one of the wedge products is almost surely zero. By \cref{le:orthogonal}, the first case is equivalent to $\spanK{K_1}$ and $\spanK{K_2}$ being orthogonal, the second case means that $K_1^{\wedge d_1}=0$ or $K_2^{\wedge d_2}=0$.
\end{proof}

\subsection{Grassmannian zonoids} Here we describe a particular subalgebra of $\A(V)$, which will play a role in our forthcoming work \cite{BBML2021}.

\begin{definition}
Let $K\in\ZZ(\Lambda^kV)$. We say that $K$ is a \emph{Grassmannian zonoid} if there is random vector $X\in \Lambda^k V$ such that $K=K(X)$ and such that $X$ is almost surely simple, i.e. $X$ almost surely takes values in the cone of simple vectors $\{v_1\wedge \cdots \wedge v_k \mid v_1,\ldots, v_k\in V\}$.
The set of Grassmannian zonoids will be denoted $\GZ(k,V)\subset \ZZ(\Lambda^k V)$ and its linear span in~$\VZ(\Lambda^k V)$ will be denoted by $\VGZ(k,V)$
\end{definition}

In the correspondence between zonoids and measures on the projective space described in \cref{sec:VZandmeas}, Grassmannian zonoids have a simple description. Indeed using the definition above and \cref{prop:fromrandtomeas} we see that a  zonoid $K\in\ZZ(\Lambda^kV)$ is Grassmannian if and only if its corresponding measure is supported on the Grassmannian, considered as the simple vectors in $\PP(\Lambda^k V)$ via the Pl\"ucker embedding $\mathrm{span}(v_1,\ldots,v_k)\mapsto [v_1\wedge \cdots\wedge v_k]$. Hence the name \emph{Grassmannian} zonoids.

\begin{proposition}\label{prop:sumandwedgeofgrass}
    The wedge product of two Grassmannian zonoids is a Grassmannian zonoid. Moreover, if they are of the same degree, the sum of two Grassmannian zonoids is a Grassmannian zonoid. Hence $\VGZ(k,V)\subset \VZ(\Lambda^k V)$ consists only of differences of Grassmannian zonoids and $\GZ(k,V)$ is a convex cone in $\VGZ(k,V)$.
\end{proposition}
\begin{proof}
Let $K=K(X_1\wedge\cdots\wedge X_p)\in \GZ(p,V)$ and $L=K(Y_1\wedge\cdots\wedge Y_q)\in \GZ(q,V)$ with $X_1\wedge\cdots\wedge X_p$ independent of $Y_1\wedge\cdots\wedge Y_q$. Then by definition of the wedge product we have $K\wedge L=K(X_1\wedge\cdots\wedge X_p\wedge Y_1\wedge\cdots\wedge Y_q)$ which is Grassmannian. Now suppose $p=q$ then by \cref{summation_formula}, we have $K+L=K(Z)$ where $Z=2(1-\epsilon) X_1\wedge\cdots\wedge X_p + 2\epsilon Y_1\wedge\cdots\wedge Y_p$ where $\epsilon$ is a Bernoulli variable of parameter $\tfrac{1}{2}$ independent of the other variables. It is then enough to see that $Z\in \Lambda^{p} V$ is almost surely simple.
\end{proof}

\begin{definition}\label{def:GA}
The \emph{Grassmannian zonoid algebra} is defined as
\begin{equation}
    \GA(V):=\bigoplus_{k=0}^{d}\VGZ(k,V).
\end{equation}
\cref{prop:sumandwedgeofgrass} guarantees that $  \GA(V)$ is a subalgebra of $\A(V)$.
\end{definition}

A particular case of Grassmannian zonoids are the \emph{decomposable} ones. Namely, if we have $K_1,\ldots,K_k\in\ZZ(V)$, then $K_1\wedge\cdots\wedge K_k\in\GZ(k,V)$ by \cref{prop:sumandwedgeofgrass}. This corresponds to the case $K_1\wedge\cdots\wedge K_k=K(X)$ with $X=X_1\wedge \cdots\wedge X_k$, where $X_1, \ldots, X_k$ are
\emph{independent} random vectors.
In general, a zonoid in $\GZ(k,V)$ represented by $X=X_1\wedge \cdots\wedge X_k$ does not need to have independent $X_i$.
The next proposition says that decomposable Grassmann zonoids span actually a dense subspace.

\begin{proposition}
    Finite sums of zonoids of the form  $K_1\wedge\cdots\wedge K_k\in\GZ(k,V)$ are dense in~$\GZ(k,V)$. Hence the set $\left\{K_1\wedge\cdots\wedge K_k \st K_1,\ldots,K_k\in\ZZ(V)\right\}$ spans a dense subspace in the virtual Grassmann zonoids $\VGZ(k,V)$.
\end{proposition}
\begin{proof}
By \cref{rk:lolnforzonoids}, any zonoid in $\GZ(k,V)$ is the limit of finite sums of segments that are of the form $\tfrac{1}{2}[-w,w]$ with $w=x_1\wedge\cdots\wedge x_k$. It is then enough to see that such segments are decomposable. Indeed we have $\tfrac{1}{2}[-w,w]=\tfrac{1}{2}[-x_1,x_1]\wedge\cdots\wedge \tfrac{1}{2}[-x_k,x_k]$.
\end{proof}

This last fact gives good hope to extend properties of decomposable zonoids to the Grassmannian ones. For example we conjecture the following.

\noindent\textbf{Conjecture.}
For any $K,L\in\ZZ(\R^m)$ and any $C\in\GZ(m-2,\R^m)$, we have
\begin{equation}
    \ell(K\wedge L \wedge C)^2\geq \ell(K\wedge K \wedge C)\ell(L\wedge L \wedge C).
\end{equation}
This would generalize Alexandrov--Fenchel, which corresponds to the case where $C$ is decomposable; i.e., of the form $C=K_1\wedge\cdots\wedge K_{m-2}$ with $K_1,\ldots,K_{m-2}\in\ZZ(\R^m)$, see \cref{AF_for_length} below.
\subsection{Hodge duality}

The exterior algebra of an oriented Euclidean vector space~$V$ comes
with the duality given by the Hodge star operation.
Let us briefly recall this notion. Upon choosing an orientation of~$V$,
$\Lambda^m V$ becomes an oriented one dimensional Euclidean vector space
that can be identified with~$\R$.
More specifically, if $e_1,\ldots,e_m$ is an oriented orthonormal basis of $V$, then
$e_1\wedge \cdots \wedge e_m$ is the distinguished generator of $\Lambda^m V$.
When $d_1+\ldots+d_p=m$, the wedge product thus induces a multilinear map
\be
 \VZ(\Lambda^{d_1}V)\times \cdots \times  \VZ(\Lambda^{d_p}V) \to  \VZ(\Lambda^{m}V) \simeq \R.
\ee
Therefore the wedge product $K_1\wedge \cdots \wedge K_p$ of zonoids,
which is a segment in $\Lambda^m V$,
can be identified with the real number giving the length of this segment using \cref{re:1-dim}.
This remark will play a role in \cref{sec:MVandL}.

The \emph{Hodge star operation} is the isometric linear map
$ \Lambda^d V \to \Lambda^{m-d} V,\: v \mapsto \star v$
characterized by
$\langle u, v\rangle = u \wedge \star v$
for all $u\in V$.
This defines an involution (up to sign) of the exterior algebra $\Lambda V$, see~\cite[p.16]{flanders}.
Via \cref{def:linearext}
we associate with the Hodge star operation the linear isomorphism
\be
\VZ\big(\Lambda^d V\big) \to \VZ\big(\Lambda^{m-d} V\big),\:
K \mapsto \star K
\ee
that we conveniently denote with the symbol $\star$ as well.
If $X \in \Lambda^k V$ is an integrable random variable,
then by \cref{lemma:linear}
\be\label{eq:KXstar}
 \star(K(X)) = K(\star X) .
\ee
This shows that, if $K\in \Z\big(\Lambda^k V\big)$ is a zonoid, then $\star K$ is a zonoid as well.
The support function of $\star K$ satisfies
\be\label{eq:hKXstar}
 h_{\star K}(\star u) =
  \tfrac{1}{2}\EE |\langle \star u, \star X\rangle| =
  \tfrac{1}{2}\EE |\langle u, X\rangle| = h_K(u) ,
\ee
where we used that $\star$ is isometric for the second equality.

\begin{proposition}\label{pro:hodge}
The \emph{Hodge star operation} on zonoids
$\VZ\big(\Lambda^d V\big) \to \VZ\big(\Lambda^{m-d} V\big)$
is a norm and order preserving linear isomorphism.
It also preserves the length of zonoids, so that for all
zonoids $K\in \Z\big(\Lambda^d V\big)$
$$
 \|K\| = \|\star K\|, \quad \ell(K) = \ell(\star K) .
$$
The Hodge star operation defines a linear involution of the zonoid algebra~$\A(V)$.
\end{proposition}

\begin{proof}
Using \cref{eq:hKXstar} and \cref{useful_properties_for_h}, we obtain
$$
\|\star K\| = \|\bar{h}_{\star K}\|_{\infty} = \|\bar{h}_{K}\|_{\infty} = \| K\| .
$$
Moreover, using \cref{mydef:length}, we get
$$
 \ell(K)= \EE \|X\| = \EE \|\star X\| = \ell(\star K) .
$$
The fact that it is an involution on the space of zonoids follows from the fact that the Hodge star is an involution up to sign and that the zonoids are centrally symmetric.
\end{proof}

\begin{remark}
$\star K$ should not be confused with the polar dual of $K\in \Z\big(\Lambda^d V\big)$,
which is a convex body living in the same space as $K$, and in general not a zonoid, see~\cite{ACOCB}.
\end{remark}

\begin{remark}
    Note that since the Hodge star operation preserves simple vectors, if $K\in\GZ(k,V)$ is a Grassmannian zonoid then $\star K \in \GZ(m-k,V)$ is also Grassmannian. In other words, the map $\star$ preserves the subalgebra $\GA(V)$ of Grassmannian zonoids.
\end{remark}

Let $K\in \Z(V)$. The Hodge dual of $K^{\wedge (m-1)}$ is a zonoid in $\Lambda^1V=V$.
We show now that it is the so called \emph{projection body} $\Pi K$ of $K$.
According to ~\cite[Section~10.9]{bible},  $\Pi K$ is the convex body
whose support function is given by $u\mapsto \|u\|\vol_{m-1}\left(\pi_u(K)\right)$,
where $\pi_u$ denotes the orthogonal projection $V\to u^\perp$.

\begin{proposition}\label{eq:proj-body}
We have $\star (K^{\wedge (m-1)})=\frac{(m-1)!}{2}\Pi K$.
\end{proposition}

\begin{proof}
We show that these zonoids have the same support function.
Let $X\in V$ be a random vector representing $K=K(X)$. By definition,
$$
 \star (K^{\wedge (m-1)}) = K(Y),\quad\text{where}\quad Y:=\star (X_1\wedge\cdots\wedge X_{m-1})
$$
and  $X_1,\ldots, X_{m-1}$ are i.i.d. copies of $X$.
The definition of the Hodge dual yields
$$|\langle \star (X_1\wedge \cdots \wedge X_{m-1}),u\rangle| = | X_1\wedge \cdots \wedge X_{m-1}\wedge u |$$
and so
$$
 \vert \langle Y,u\rangle\vert = | X_1\wedge \cdots \wedge X_{m-1}\wedge u |
 = \|\pi_u(X_1)\wedge \cdots \wedge \pi_u(X_{m-1})\|\cdot \|u\|.
 $$
By \cref{eq:hE-form}, we get
$$
 h_{K(Y)}(u) = \tfrac{1}{2}\EE \vert \langle Y,u\rangle\vert =
 \tfrac{1}{2} \EE \|\pi_u(X_1)\wedge \cdots \wedge \pi_u(X_{m-1})\|\cdot \|u\|.
$$
\cref{thm:FTZCdet} from the next section,
applied to the space $u^\perp\simeq \R^{m-1}$, yields
$$
 h_{K(Y)}(u) =\frac{(m-1)!}{2}\,\vol_{m-1}(\pi_u(K))\, \|u\|.
$$
This shows that $\star (K^{\wedge (m-1)})$ has the same support function as $\frac{(m-1)!}{2}\Pi K$.
\end{proof}

We close this section by giving an interesting property of Hodge duals.
It concerns \emph{orthogonal zonoids} and will be of relevance in our upcoming work \cite{BBML2021}.

\begin{corollary}
Let $K,L\in \Lambda^{d}V$ be zonoids and denote by $\spanK{K}$ and $\spanK{L}$ their linear spans. Then
$$
 K \wedge \star L = 0 \quad \Longleftrightarrow \quad \spanK{K}\perp \spanK{L}.
$$
\end{corollary}

\begin{proof}
Let $X$ and $Y$ be integrable random vectors with $K=K(X)$ and $L=K(Y)$.
Then $K \wedge \star L = K(X\wedge \star Y)$.
By the definition of the Hodge dual, we have $X\wedge \star Y = \langle X,Y\rangle$.
Therefore,  $K \wedge \star L =0$ if and only if $\langle X,Y\rangle=0$ almost surely.
By \cref{le:orthogonal}, this is equivalent to $\spanK{K}\perp \spanK{L}$.
\end{proof}

\section{Mixed volumes and random determinants}\label{sec:MV_real}

In this section we show how the the classical notion of \emph{mixed volume} fits into our framework.
We then apply our theory to study expected absolute determinants of random matrices.
Again we denote by $V$ a Euclidean vector space of dimension~$m$.

We recall that the mixed volume is the Minkowski multilinear, translation invariant
and continuous map
$\MV\colon\K(V)^m\to \R$, defined on an $m$-tuple $K_1\ldots,K_m$ of convex bodies by
\be \label{eq:MV}
\MV(K_1, \ldots, K_m)
 =\frac{1}{m!}\frac{\partial}{\partial t_1}\cdots\frac{\partial}{\partial t_m}
 \vol_m(t_1K_1+\cdots+t_mK_m)\bigg|_{t_1=\cdots=t_m=0} ,
\ee
see \cite[Theorem 5.1.7]{bible}. For instance,
if $K_i= \sg{v_i}$ are segments, then
$$
\MV(\sg{v_1}, \ldots, \sg{v_m})
$$
equals the volume of the parallelotope spanned by $v_1, \ldots, v_m$, divided by $m!$.
The volume of $K$ is $\vol_m(K)=\MV(K, \ldots, K)$, which means that
the mixed volume is obtained from the volume function of a single body by polarization.

It is a key insight that the mixed volume of zonoids equals the length of their wedge product,
up to a constant factor.

\begin{theorem}[The wedge product of zonoids]\label{thm:FTZCdet}
Let $\widehat{\det}\colon \VZ(V)^m\to \VZ^1$ denote the multilinear map associated to the multilinear
determinant map $\det\colon V^m\to \R$ via \cref{thm:FTZC}.
Then, for every $K_1, \ldots, K_m\in \ZZ(V)$, we have
$$
 \widehat{\det}(K_1, \ldots, K_m) = K_1\wedge \cdots \wedge K_m
$$
and
$$\MV(K_1, \ldots, K_m)= \frac{1}{m!}\,\ell (K_1\wedge \cdots \wedge K_m).$$
\end{theorem}

\begin{proof}
The first identity is immediate from the definition of associated multilinear map $\widehat{\det}$.

For the second identity, we observe that
both sides are Minkowski multilinear and continuous maps $\Z(V)^m \to \R$
(for the right hand side, use \cref{th:lengthLF}).
It therefore suffices to verify the identity for segments.
So let $v_1, \ldots, v_m\in V$. Note that by \cref{thm:FTZC},
$$
L:=\sg{v_1}\wedge \cdots\wedge \sg{v_m} = \tfrac{1}{2}[-\det(v_1, \ldots, v_m), \det(v_1, \ldots, v_m)] ,
$$
This is an intervall in $\Lambda^m V\simeq\R$ and hence
$\ell(L)$ equals the volume of the parallelotope spanned by $v_1, \ldots, v_m$.
On the other hand, this equals $\MV(\sg{v_1}, \ldots, \sg{v_m})$, divided by $m!$.
This verifies the identity for segments and completes the proof.
\end{proof}

The previous theorem implies that for $K\in \Z(V)$
\be\label{ex:wB}
 \vol_m(K) = \frac{1}{m!}\, \ell(K^{\wedge m}) ,
\ee
where, again, $K^{\wedge m} = K\wedge\cdots\wedge K$ with $m$ factors.

\subsection{Length functional and intrinsic volumes}\label{sec:MVandL}

Recall \cite{klain-rota:97,bible} that the \emph{$d$-th intrinsic volume} $V_d(K)$ of a zonoid $K$ is defined as
\be
 V_d(K) := \frac{{m\choose d}}{\vol_{m-d}(B_{m-d})} \, \MV (K[d],B_m[m-d]) .
\ee
In the following, we show that the length of a zonoid, introduced in \cref{mydef:length},
is nothing but its {first intrinsic volume} (see~\cite{klain-rota:97,bible}), and that the higher intrinsic volumes can also be expressed using the length. As before, we denote by $B=B(V)$ the unit ball in $V$ and~$B^m:=B(\R^m)$.

\begin{theorem}\label{th:IVL}
The $d$th intrinsic volume of a zonoid $K\in\Z(V)$ is given by
\be
 V_d(K) = \frac{1}{d!}\, \ell(K^{\wedge d}),
\ee
where, as before, $K^{\wedge d}$ is the wedge product of $d$ copies of $K$. In particular, $$V_1(K) = \ell(K).$$
\end{theorem}

\begin{proof}
Suppose $X_1,\ldots,X_d, Y_1,\ldots,Y_{m-d}$ are independent random vectors
with values in $V$ such that the $X_i$ represent $K$ and the $Y_j$ are standard Gaussian.
Recall that $B_m = \sqrt{2\pi} K(Y_j)$.
By \cref{thm:FTZCdet}, we can write
\be
 \MV (K[d],B_m[m-d]) = \frac{(2\pi)^{\frac{m-d}{2}}}{m!}\, \EE | X_1\wedge \cdots\wedge X_d\wedge Y_1\wedge \cdots\wedge Y_{m-d} | .
\ee
We first integrate over the $Y_j$ while leaving the $X_i$ fixed, thus
\be
 \MV (K[d],B_m[m-d]) = \frac{(2\pi)^{\frac{m-d}{2}}}{m!}\, \mean_{X_i} \, \mean_ {Y_j}
   | X_1\wedge \cdots\wedge X_d\wedge Y_1\wedge \cdots\wedge Y_{m-d} | .
\ee
By orthogonal invariance of $Y:=Y_1\wedge \cdots\wedge Y_{m-d}$,
we can assume that in the inner expectation the space spanned by the $X_i$
is the span of a fixed orthonormal frame $e_1,\ldots,e_d$. Then,
$X_1\wedge \cdots \wedge X_d=\Vert X_1\wedge \cdots \wedge X_d\Vert\,  e_1\wedge \cdots \wedge e_d$
and so, using that $Y$ is independent of the $X_i$:
$$
 \MV (K[d],B_m[m-d]) =  \frac{c}{m!} \ \mean_{X_i} \Vert X_1\wedge \cdots \wedge X_d\Vert = \frac{c}{m!}\, \ell(K),
 $$
with the constant $c := (2\pi)^{\frac{m-d}{2}} \mean_Y \Vert e_1\wedge \cdots \wedge e_d\wedge Y\Vert$.
In order to determine this constant, we use that
$
 \Vert e_1\wedge \cdots \wedge e_d\wedge Y\Vert =  \Vert \tilde{Y_1} \wedge \ldots \wedge \tilde{Y_d}\Vert ,
$
where $\tilde{Y_j}$ denotes the orthogonal projection of~$Y_j$
onto the orthogonal complement $\R^{m-d}$ of $\R^d=\Span\{e_1,\ldots,e_d\}$.
Since the unit ball $B_{m-d}$ is represented by $\sqrt{2\pi}\tilde{Y_j}$, we obtain with~\cref{ex:wB},
$$
 (2\pi)^{\frac{m-d}{2}}\,\EE |\tilde{Y}_1\wedge \cdots\wedge \tilde{Y}_{m-d} | = \ell\big( B_{m-d}^{\wedge (m-d)}\big) = (m-d)! \vol_{m-d}(B_{m-d}) .
$$
We therefore conclude that
$$
 \MV (K[d],B_m[m-d]) = \frac{1}{m!} \vol_{m-d}(B_{m-d}) \ell (K^{\wedge d}),
$$
which finishes the proof.
\end{proof}

An inspection of the proof of \cref{th:IVL} reveals the following general insight.

\begin{corollary}\label{cor:LL}
Let $L$ be a zonoid in $\Lambda^e V$,
which is invariant under the action of the orthogonal group $O(V)$
(for instance, $L =B^{\wedge e}$).
Further, let $d+e \le m$.
Then there exists a constant $c_{d,m}(L)$,
only depending on $d,m$ and $L$, such that
$$
 \ell (K \wedge L) = c_{d,m}(L)\ \ell(K)
 $$
for any zonoid $K=K_1\wedge \cdots\wedge K_d$ with $K_1, \ldots,K_d \in \ZZ(V)$.\qed
\end{corollary}

The \emph{Alexandrov--Fenchel inequality} for convex bodies is one of deepest results of the
Brunn-Minokowski theory. Via \cref{thm:FTZCdet}, we can express this inequality for zonoids
in terms of the length as follows:
if $K_1,K_2,\ldots,K_m\in\ZZ(V)$ are zonoids in $V$, then
\be\label{AF_for_length}
\ell(K_1\wedge K_2\wedge K)^2 \geq \ell(K_1\wedge K_1\wedge K)\cdot \ell(K_2\wedge K_2\wedge K),
\text{ where }  K = K_3\wedge \cdots \wedge K_m;
\ee
see, e.g., \cite[Theorems 6.3.1]{bible}.
More generally, the \emph{general Brunn-Minkowski theorem} \cite[Theorem 6.4.3]{bible}
implies that for any $1\leq d\leq m$:
\be\label{BM_for_length}
t\mapsto \ell(K_t^{\wedge d} \wedge K_{d+1}\wedge \cdots\wedge K_m)^\frac{1}{d} \quad\text{ is concave for $t\in[0,1]$},
\ee
where $K_t:=tK_1 + (1-t)K_2$.

Using \cref{cor:LL}, we deduce from \cref{AF_for_length} the following special case:
\be\label{AF_for_length_2}
 \ell(K \wedge L)^2 \geq \ell(K\wedge K)\ \ell(L\wedge L)
\ee
for all zonoids $K,L\in \ZZ(V)$.
Moreover, \cref{BM_for_length} means that
\be\label{BM_for_length_2}
t\mapsto \ell((tK + (1-t)L)^{\wedge d})^\frac{1}{d} \quad\text{ is concave for $t\in[0,1]$.}
\ee
By \cref{cor:LL}, we can generalize \cref{AF_for_length_2} by replacing $K\wedge L$ with the wedge product of $K\wedge L$ with any orthogonally invariant zonoid in $M\in \ZZ(\Lambda^k(V))$ such that $k+d\leq m$, and get
$\ell(K \wedge L\wedge M)^2 \geq \ell(K\wedge K\wedge M)\ \ell(L\wedge L\wedge M)$.
Similarly, in \cref{BM_for_length_2} we may also take the product of the $d$-th wedge power with $M$
and obtain that the function
$\ell((tK + (1-t)L)^{\wedge d}\wedge M)^\frac{1}{d}$ is concave for $t\in[0,1]$.

\subsection{Random determinants}

The purpose of this section is to generalize a result due to Vitale.
In \cite{Vitale} Vitale showed that if $X\in \R^m$ is an integrable random vector and
$M_X$ is the $m\times m$ matrix whose columns are i.i.d. copies of $X$, then $\EE|\det(M_X)|=m!\vol_m(K(X))$.
We generalize this result to independent blocks that can give different distributions. This is to be compared with \cref{thm:vitaleC} below, in which we prove a similar result, but for complex random matrices.

\begin{theorem}[Expected absolute determinant of independent blocks]\label{thm:genvitale}
Let $M=(M_{1}, \ldots, M_p)$ be a random $m\times m$ matrix partitioned into blocks $M_j$ of size $m\times d_j$,
with $d_1+\cdots +d_p=m$.
We denote by $v_{j,1}, \ldots, v_{j,d_j}$ the columns of $M_j$ and assume that
$Z_j := v_{j,1}\wedge \cdots \wedge v_{j, d_j}\in \Lambda^{d_j}\R^m$ is integrable.
If the random vectors $Z_1\in\Lambda^{d_1}(\R^m), \ldots, Z_p\in \Lambda^{d_p}(\R^m)$
are independent, then:
$$
 \EE |\det( M)| = \ell(K(Z_1)\wedge \cdots\wedge K(Z_p)).
 $$
In particular, if $p=m$ and $d_1=\cdots =d_p=1$, then
$$
 \EE |\det( M)| = m!\, \MV(K(Z_1),\ldots, K(Z_m)).
$$
\end{theorem}

The formula for independent columns, i.e., where $p=m$, was already proved by Weil in \cite[Theorem 4.2]{Weil1976}. In fact, Weil proved a more general version in the case $p=m$ for convex bodies, not just zonoids.

Vitale's theorem corresponds to the special case $d_j=\cdots=d_m=1$ and where the matrices $M_1, \ldots, M_m$
(which are now column vectors) all have the same distribution: $M_j\sim X$. In this case,
$\EE|\det(M)|=m! \MV(K(X),\ldots, K(X)) = m! \vol_m(K(X))$.

\begin{proof}[Proof of \cref{thm:genvitale}]
Let us first observe that
$|\det[v_1, \ldots, v_m]|=\|v_1\wedge \cdots\wedge v_m\|$ for vectors $v_1, \ldots, v_m\in \R^m$.
Applying this in our case, we get for the absolute determinant of $M$ that
$ |\det(M)|=\|v_{1, 1}\wedge\cdots \wedge v_{p, a_p}\|=\|Z_1\wedge \cdots \wedge Z_p\|.$
Taking expectations on both sides,
we obtain
$\EE|\det(M)|=\ell(K(Z_1\wedge\cdots\wedge Z_p))$.
By the definition of the wedge product, we have
$K(Z_1\wedge\cdots\wedge Z_p) = K(Z_1)\wedge \cdots \wedge K(Z_p)$,
which shows the first assertion.
The second claim follows from \cref{thm:FTZCdet}. This concludes the proof.
\end{proof}

We give two additional examples in which \cref{thm:genvitale} is applied.

\begin{example}\label{ex:genvitale}
Let $Z_1, \ldots, Z_n\in \C^n$ be integrable random vectors and $L\in \C^{n\times n}$
be the random matrix $L=(Z_1, \ldots, Z_n)$.
We show how to compute $\EE|\det (L)|^2$ with \cref{thm:genvitale}
(in the next section we will explain how to compute $\EE|\det (L)|$; see \cref{thm:vitaleC}).
To this end,
we decompose $Z_j=X_j+\sqrt{-1}Y_j$ with real random vectors $X_j, Y_j\in \R^n$ (possibly dependent),
We put $m:=2n$ and
consider the random matrix $M=(M_1, \ldots, M_n)$,
where $M_j=\left(\begin{smallmatrix}X_j & -Y_j\\ Y_j & X_j\end{smallmatrix}\right)$,
which satisfies the hypothesis of \cref{thm:genvitale}.
Observe that $|\det(L)|^2= |\det(M)|$.
If we define the integrable random vector
\be
Q_j:=\left(\begin{matrix}X_j\\
Y_j\end{matrix}\right)\wedge \left(\begin{matrix}-Y_j\\
X_j\end{matrix}\right)\in \Lambda^2(\R^{m}), 
\ee
we get by \cref{thm:genvitale} that
$\EE|\det L|^2= \EE|\det(M)|
=\ell(K(Q_1)\wedge \cdots \wedge K(Q_n ))$.
\end{example}

\begin{example}\label{ex:parallelotope}
We interpret here \cref{thm:genvitale} geometrically as follows:
we consider a random parallelotope $P\subset \R^m$ spanned by $k\leq m$ random vectors, and ask for
its expected $k$--dimensional volume. Suppose that the random vectors are $v_{i,j}$ for
$1\leq i\leq p$ and $1\leq j\leq d_i$ with $\sum_{j=1}^p{d_j}=k$ and such that $v_{i_1,j_1}$ and
$v_{i_2,j_2}$ are independent, if $i_1\neq i_2$.
Setting $Z_i:= v_{i,1}\wedge\cdots \wedge v_{i,d_i}$, we therefore have
$Z_1\wedge \cdots \wedge Z_p\in \Lambda^{k}(\R^m)$ and
$$
 \EE \left(\vol_k(P)\right)= \ell(K(Z_1)\wedge \cdots \wedge K(Z_p)) .
$$
This shows that the length functional can be used to compute the expected volume of $P$.
\end{example}

\cref{thm:genvitale} has an interesting consequence when combined with the {Alexandrov--Fenchel inequality}~\cref{AF_for_length}
and the {general Brunn-Minkowski theorem} \cref{BM_for_length}.

\begin{corollary}[Brunn--Minkowski theorem for expected determinants]\label{cor_AF}
Let $X_1,\ldots,X_m$ be independent integrable random vectors in $\mathbb R^m$,
and let $X_1'\sim X_1$ and $X_2'\sim X_2$ be independent of $X_1,X_2$, respectively. Then:
$$
\big(\mean  \vert \det
[\begin{smallmatrix}
X_1 &X_2& X_3& \ldots& X_m
\end{smallmatrix}]\vert\big)^2 \geq
\mean  \vert \det
[\begin{smallmatrix}
X_1 &X_1'& X_3& \ldots& X_m
\end{smallmatrix}]\vert
\cdot \mean  \vert \det
[\begin{smallmatrix}
X_2 &X_2'& X_3& \ldots& X_m
\end{smallmatrix}]\vert.
$$
More generally, for any $1\leq d\leq m$, the following function is concave for $t\in[0,1]$:
$$t\mapsto \big(\mean  \vert \det
[\begin{smallmatrix}
X^{(t)}_1 &\ldots & X^{(t)}_d & X_{d+1}& \ldots& X_n
\end{smallmatrix}]\vert\big)^\frac{1}{d},$$
where $X^{(t)}_1,\ldots,X^{(t)}_d$ are independent copies of
$t\epsilon 2 X_1+ (1-t)(1-\epsilon) 2 X_2$ and $\epsilon$ is a
Bernoulli random variable with success probability $\tfrac{1}{2}$,
which is independent of $X_1$, $X_2$, $X_{d+1},\ldots, X_n$.
\end{corollary}

\begin{proof}
We set $K_i:=K(X_i)$ for $1\leq i\leq m$.
The first inequality is \cref{AF_for_length} combined with \cref{thm:genvitale}.
For proving the second statement we let $K_t:=tK_1 + (1-t)K_2$.
Then, by \cref{summation_formula}, $K_t=K(X^{(t)})$, where
$X^{(t)} =  t\epsilon 2 X_1+ (1-t)(1-\epsilon) 2 X_2$.
We combine \cref{thm:genvitale} with \cref{BM_for_length} to conclude.
\end{proof}

We find this a quite remarkable consequence, since it holds under very weak assumptions:
only independence and integrability of the $X_i$ is assumed.
To our best knowledge, this is a new formula that has not been described in the literature.

\begin{remark}
	The equality case of Alexandrov--Fenchel for zonoids was described in~\cite{SchAF}. From this, one can deduce the equality case for random determinants in \cref{cor_AF}.
\end{remark}

\begin{example}
We can combine \cref{AF_for_length_2} and \cref{BM_for_length_2}
with \cref{thm:genvitale}
to obtain a result about expected volumes of random triangles
(this is the case $k=2$ in \cref{ex:parallelotope}):
\be\label{AF_for_parallelotopes}
\big(\mean  \vol_2(\Delta(X,Y))\big)^2 \geq
\mean  \vol_2(\Delta(X,X'))
\cdot \mean  \vol_2(\Delta(Y,Y')),
\ee
where $X,Y,X',Y'$ are independent vectors with finite expected norm, $X\sim X'$ and $Y\sim Y'$,
and $\Delta(X,Y)$ is the triangle, whose vertices are the origin and $X$ and $Y$.
We also get that the function
$$
t\mapsto \big(\mean  \vol_d(P(
X^{(t)}_1,\ldots, X^{(t)}_d))\big)^\frac{1}{d} \quad \text{is concave in $t\in[0,1]$},
$$
where the $X^{(t)}_i$ are defined as in \cref{cor_AF} and $P(
X^{(t)}_1,\ldots, X^{(t)}_d)$ is the parallelotope spanned by the $X^{(t)}_i$.
\end{example}

\section{Mixed $J$--volume and random complex determinants}\label{sec:mixed_J_volume}

In this section $V$ will be a complex vector space of complex dimension $n$, that is a real vector space of real dimension $2n$ together with a real linear endomorphism $J:V\to V$ such that $J^2=-\textbf{1}$. $J$ will be called the \emph{complex structure} of $V$. We recall that such a complex structure induces an isomorphism $V\cong \C^n$ under which the automorphism $J$ corresponds to multiplication by $i$.

Here we will introduce a notion similar to mixed volume for zonoids in $V$, which  is  adapted to the complex structure. We call it the \emph{mixed $J$--volume} and denote it by $\MV^J$. It takes $n$ zonoids in a $2n$--dimensional real vector space, while the ordinary mixed volume $\MV:\ZZ(V)^{2n}\to \R$ is instead a function of $2n$ arguments. Furthermore, $\MV^J$ is Minkowski additive and positively homogeneous in each argument; see \cref{def:mixJ} below. We have already seen an application of our theory to zonoids in a complex vector space in \cref{ex:genvitale}. This example, however, used real multilinear maps. In this section, we consider \emph{complex multilinear maps}, which leads to a different notion.

There are some key properties that make the mixed $J$--volume interesting: (1) it is compatible with the complex structure (see \cref{propo:MVJ} (3)); it allows to formulate a complex version of Vitale's theorem, for computing the expectation of the \emph{modulus} -- and not the \emph{modulus squared}, as it is done in \cref{ex:genvitale} -- of the determinant of a random $n\times n$ matrix with rows which are independent random variables in $\C^n$ (\cref{thm:vitaleC}); (3) it can be defined on all polytopes (but does not continuously extends  all convex bodies, see \cref{cor:noextofJ}); and finally (4) it equals the classical mixed volumes when restricted to polytopes in~$\mathbb R^n \subset \mathbb C^n$;
see \cref{propo:MVJ} (2).

The extension of the $J$--volume to polytopes in $V\simeq \C^n$ is a so called \emph{valuation}; see \cref{def:val}. The proof of the extension of the $J$--volume to polytopes is especially interesting and uses some particular combinatorial properties of zonotopes, and it connects to similar notions, such as Kazarnovskii's pseudovolume (see \cref{def:Kaza} below).

 \subsection{The mixed $J$--volume of zonoids}We endow our complex vector space $(V, J)$ with a hermitian structure $\phi:V\times V\to \C$.  The associated scalar product is the real part of~$\phi$. When $V=\C^n$ we consider the standard complex structure where $J$ is the multiplication by $i=\sqrt{-1}$ and the standard hermitian structure. Moreover, for $0\leq k\leq n$ we denote by  $\Lambda_\mathbb{C}^k(V)$ the complex exterior algebra and, given vectors $v_1, \ldots, v_k\in V$, we denote  by $v_1\wc\cdots \wc v_k\in \Lambda^k_\C(V)$
 their complex exterior product. Note that this construction depends on the choice of the complex structure $J$, however we prefer the notation with ``$\C$'' that we find easier to read.

The hermitian structure on $V$ induces an hermitian structure on all the complex exterior powers and, in particular, taking its real part, a real scalar product on each of them. This implies that we have a Euclidean norm on each $\Lambda^k_\C(V)$ and consequently we have a length functional $\ell:\Z(\Lambda^k_\C(V))\to \R$; see \cref{mydef:length}.

Moreover the complex wedge product is, in particular, a real multilinear map. Therefore we can apply \cref{thm:FTZC} to obtain  a well-defined notion of complex product of virtual zonoids.
\begin{definition}		\label{def:complex_wedge}
	Consider the (real) multilinear map $F\colon V^n\to \Lambda^n_\C(V)$ defined by the complex wedge $F(v_1, \ldots, v_n):=v_1\wc\cdots \wc v_n$. For any  $K_1, \ldots, K_n\in \VZ(V)$, we define:
	\begin{equation}
		K_1\wc\cdots \wc K_n:=\widehat{F}(K_1, \ldots, K_n) .
	\end{equation}
\end{definition}
The next definition uses this construction to define the mixed $J$--volume.

\begin{definition}[Mixed $J$--volume]\label{def:mixJ}
 We define the \emph{mixed $J$--volume}  $\MV^J:\VZ(V)^n\to \R$ to be the $\R$--multilinear map given, for all $K_1, \ldots, K_n\in \VZ(V)$, by:
	\begin{equation}
		 \MV^J(K_1, \ldots, K_n):=\frac{1}{n!}\ell\left(K_1\wc \cdots \wc K_n\right).
	\end{equation}
The \emph{$J$--volume} of a zonoid $K\in \VZ(V)$ is defined to be:
\be \vol_n^J(K):=\MV^J(K, \ldots, K).\ee
\end{definition}

\begin{remark}
	Notice that, since $\Lambda^{2n}( V)\simeq \R$ is of real dimension one, zonoids in $\Lambda^{2n}(V)$ are just segments. By contrast, the top complex exterior power $\Lambda_\C^n (V)\simeq \C$ is of real dimension two and centered zonoids in this space are more than segments (in fact they are precisely the centrally symmetric convex bodies; see \cite[Theorem 3.5.2]{bible}). Thus $K_1\wc \cdots \wc K_n$ is a zonoid in $\Lambda_\C^n (V)\simeq \R^2$. Then taking its length loses some information. However, it is easy to see using \cref{def:complex_wedge}, that if one of the $K_i$ is invariant under the $U(1)$ action on $V$, then $K_1\wc \cdots \wc K_n$ is also $U(1)$ invariant and hence must be a disc. We compute the length of a disc in \cref{lemma:D} below.
\end{remark}

Let us study some of the properties of the mixed $J-$volume. On some classes of zonoids of the complex space $V$ it behaves particularly well. The first case is when $V=\C^n$ and all the zonoids are contained in the real $n-$plane $\R^n\subset \C^n$. In that case, we will show that the mixed $J-$volume is equal to the classical mixed volume (see \cref{propo:MVJ} (2)).

Next, we consider complex discs.
\begin{definition}
	Let $z\in V$. We define $D_z$ to be the closed centered disc of radius $\vert z\vert$ in the complex line $\C z \cong \mathbb R^2$.
\end{definition}
In order to describe a random vector representing $D_z$, let us introduce the following notation. For $\theta\in \R$ we denote by $e^{\theta J}:V\to V$ the linear operator
$
e^{\theta J}:=\cos(\theta) \mathrm{Id}+\sin (\theta)  J
$
where $\mathrm{Id}$ denotes the identity on $V$.
We then have the following lemma.

\begin{lemma}\label{lemma:D}
Let $\theta\in [0, 2\pi]$ be a uniformly distributed random variable and $z\in V$ nonzero.
Consider the random vector $X_z\in V$ defined by
$ X_z:=\pi e^{\theta J}z.$ Then:
$$ K(X_z)=D_z\quad \text{and}\quad\ell(D_z)=\pi \|z\|.$$
\end{lemma}
\begin{proof}
Since for every $\theta \in [0, 2\pi]$ the vector $e^{\theta J}z$
belongs to $\C z$, we have  $h_{K(X_z)}(u)=0$ for every $u\in (\C z)^\perp.$
This implies that $K(X_z)$ is contained in $\C z$.
It is straightforward to verify that
$
 \EE |\langle e^{\theta J} z, z\rangle| = \|z\|^2 \, \EE |\cos\theta| = \frac{2}{\pi} \|z\|^2 .
$
This implies for $\lambda\in \C$ that
\begin{align}
h_{K(X_z)}(\lambda z) = \frac12 \EE |\langle X_z, \lambda z \rangle|
 = \frac12 \pi |\lambda| \, \EE |\langle e^{\theta J} z, z\rangle| = |\lambda| \|z\|^2
 = \|z\| \cdot |\lambda z| .
\end{align}
On the other hand,
$h_{D_z}(\lambda z) = \|z\| \cdot |\lambda z|$,
hence the first assertion follows. The second statement follows immediately from the fact that
$\|X_z\|=\pi \|z\|$ almost surely.
\end{proof}


\begin{proposition}[Properties of the mixed $J$--volume]\label{propo:MVJ}The following properties hold:
	\begin{enumerate}
		\item The mixed $J$--volume of zonoids $\MV^J\colon \Z(V)^n\to \R$ is symmetric, multilinear, and monotonically increasing in each variable.
		\item Suppose $V=\C^n$ and let $K_1,\ldots,K_n\in \Z( \mathbb R^n)\subset \Z(\mathbb C^n)$. Then:
		\begin{equation}
		\MV^J(K_1, \ldots, K_n) = \MV(K_1, \ldots, K_n).
		\end{equation}
		\item Let $T:V\to V$ be a $\C$--linear transformation (i.e., such that $TJ=JT$), and denote by $\det_\C(T)$ its complex determinant. Then, for all $K_1, \ldots, K_n\in \Z(V)$,
		\be \MV^{J}(TK_1, \ldots, TK_n)=|{\det}_\C(T)|\, \MV^{J}(K_1, \ldots, K_n) . \ee
		\item For every $z_1, \ldots, z_n\in V$ we have
		$ \MV^J(D_{z_1},\cdots, D_{z_n})=\frac{\pi^n}{n!}|z_1\wc\cdots \wc z_n|$
		\item For every $\theta\in\R$ and every $K_1,\ldots,K_n\in\Z(V)$ we have
		\begin{equation}
		\MV^J(e^{\theta J}K_1,K_2,\ldots, K_n)=\MV^J(K_1,\ldots, K_n) .
		\end{equation}
	\end{enumerate}
\end{proposition}
\begin{proof} Multilinearity of $\MV^J$ follows from the definition and \cref{thm:FTZC}. To see that $\MV^J$ is symmetric, given zonoids $K_1, \ldots, K_n$ in $V$, let $X_1, \ldots, X_n\in V$ be independent integrable random vectors such that 
$K_j=K(X_j)$. We have
	\begin{align} K_1\wc K_2\wc\cdots \wc K_n&=K(X_1\wc X_2\wc\cdots \wc X_n)\\
	&=K(-X_2\wc X_1\wc\cdots \wc X_n)\\
	&=K(X_2\wc X_1\wc\cdots \wc X_n)&\textrm{(by \cref{lemma:scaling})}\\
	&=K_2\wc K_1\wc\cdots \wc K_n.\end{align}
	The same argument gives symmetry in each pairs of variables. The fact that the mixed $J$--volume is monotonically increasing in each variable is a direct cosequence of the definition, \cref{thm:FTZC} and the monotonicity of the length (\cref{cor:monolength}) .

	Let us prove point (2). Let $K_1,\ldots,K_n\in \Z( \mathbb R^n)\subset \Z(\mathbb C^n)$ and let $X_1,\ldots,X_n\in\R^n$ be independent random (real) vectors such that $K_j=K(X_j)$, $1\leq j\leq n$. By \cref{def:mixJ} the mixed $J$--volume is
  $ \MV^J(K_1, \ldots, K_n)= \tfrac{1}{n!}\ell\left(K_1\wc \cdots \wc K_n\right) = \tfrac{1}{n!}\mean  \Vert X_1\wc \cdots \wc X_n\Vert$.
  Because the $X_j$ are real vectors,
  the span of the $X_j$ defines a \emph{Lagrangian plane} (a plane $E$ that is orthogonal to $JE$), and so
  we have  $\Vert X_1\wc \cdots \wc X_n\Vert = \Vert X_1\wedge \cdots \wedge X_n\Vert$,
  by \cref{lemma:Ej}.
  We get $\MV^J(K_1, \ldots, K_n) = \tfrac{1}{n!}\mean  \Vert X_1\wedge \cdots \wedge X_n\Vert$. We conclude from \cref{thm:FTZCdet} that the latter is equal to $\MV(K_1, \ldots, K_n)$. This finishes the proof of point (2).

	In order to prove point (3), let $K_1,\ldots,K_n\subset V$ be zonoids and let again $X_j\in V$ be random vectors such that $K_j=K(X_j)$. Then $TK_j=K(TX_j)$ and we have
	\begin{align} \MV^{J}(TK_1, \ldots, TK_n)&=\frac{1}{n!}\EE|TX_1\wc \cdots \wc TX_n|\\
	&=\frac{1}{n!}\EE|({\det}_\C(T))X_1\wc \cdots \wc X_n|\\
	&=\frac{1}{n!}|{\det}_\C(T)|\EE|X_1\wc \cdots \wc X_n|\\
	&=|{\det}_\C(T)|\MV^J(K_1, \ldots, K_n).\end{align}

	To show point (4) we use \cref{lemma:D} and write, for $\theta_1, \ldots, \theta_n\in [0,2\pi]$ independent and uniformly distributed:
	\begin{align} \MV^J(D_{z_1}, \ldots, D_{z_n})&=\frac{1}{n!}\ell(D_{z_1}\wc \cdots \wc D_{z_n})\\
	&=\frac{1}{n!}\ell\bigg(K(\pi e^{J\theta_1}z_1)\wc\cdots \wc K(\pi e^{J\theta_n}z_n)\bigg)\\
	&=\frac{1}{n!}\EE\left| (\pi e^{J\theta_1}z_1)\wc\cdots \wc (\pi e^{J\theta_n}z_n)\right|=\frac{\pi^n}{n!}|z_1\wc\cdots \wc z_n|,
	\end{align}
	which is what we wanted.

	Finally, to show the last item, let again $X_j\in V$ be random vectors such that $K_j=K(X_j)$. Then $e^{\theta J}K_j=K(e^{\theta J}X_j)$ and we have
	\begin{align} \MV^{J}(e^{\theta J}K_1, K_2, \ldots, K_n)&=\frac{1}{n!}\EE|e^{\theta J}X_1\wc X_2 \wc \cdots \wc X_n|\\
	&=\frac{1}{n!}\EE|X_1\wc \cdots \wc X_n|=\MV^{J}(K_1, \ldots, K_n).
	\end{align}
	This concludes the proof.
\end{proof}

\begin{remark}
	It is unknown to us if the mixed $J$--volume satisfies an \emph{Alexandrov--Fenchel inequality} and we leave this for future work.
\end{remark}

\subsection{Random complex determinants}
Here we state and prove a complex version of \cref{thm:genvitale} which gives a way to describe the expectation of the modulus of the determinant of a random complex  matrix with independent blocks (in \cref{ex:genvitale} we used \cref{thm:genvitale} to compute instead the expectation of the \emph{square} of the modulus of the determinant). To do that, mimicking the definition of \cref{section:ZonoidAlgebra}, we associate with each complex wedge product
$\wc\colon\Lambda_\C^d V\times \Lambda_\C^e V\to \Lambda_\C^{d+e} V$
the componentwise continuous bilinear map
\begin{equation}
\wc : \VZ\big(\Lambda_\C^d V\big)\times \VZ\big(\Lambda_\C^e V\big)\to \VZ\big(\Lambda_\C^{d+e} V\big).
\end{equation} induced from it by \cref{thm:FTZC}. We then have the following.

\begin{theorem}\label{thm:vitaleC}
Let $M=(M_{1}, \ldots, M_p)\in \C^{n\times n}$ be a random complex $n\times n$ matrix partitioned into blocks
$M_j$ of size~$n\times d_j$, with $d_1+\cdots +d_p=n$.
For $j=1, \ldots, n$, denote by $v_{j,1}, \ldots, v_{j,d_j}$ the columns of $M_j$ and
assume that
$Z_j=v_{j,1}\wedge \cdots \wedge v_{j, d_j}\in \Lambda_\C^{d_j}\C^n$ is integrable.
If the random vectors $Z_1\in\Lambda_\C^{d_1}(\C^n), \ldots, Z_p\in \Lambda_\C^{d_p}(\C^n)$ are independent, then:
$$ \EE |\det( M)|=\ell(K(Z_1)\wedge_\C \cdots\wedge_\C K(Z_p)).$$
In particular, if $p=n$ and $a_1=\cdots =a_p=1$, then
$$\EE |\det( M)| = n!\, \MV^J(K(Z_1),\ldots, K(Z_n)).$$
\end{theorem}

\begin{proof}
Recall from \cref{def:mixJ} that
\be n!\,\MV^J (K(Z_1), \ldots, K(Z_n)) = \ell(K(Z_1)\wc \cdots \wc K(Z_n)).\ee
Moreover we have $\ell(K(Z_1)\wc \cdots \wc K(Z_n)) = \EE|Z_1\wc \cdots \wc Z_n|$.
The last term is equal to $\EE|\det (M)|$. This concludes the proof.
\end{proof}

\begin{remark}
	Notice that, in the case where $p=n$ and $d_1=\cdots =d_p=1$ and if the random matrix is almost surely real, \cref{thm:vitaleC} agrees with \cref{thm:genvitale}, since $\MV$ and $\MV^{J}$ coincide on \emph{real} zonoids; see \cref{propo:MVJ} (2).
\end{remark}

As an application of \cref{thm:vitaleC}, we compute the $J$--volume of balls.
\begin{corollary}\label{eq:Jvolball}The $J$--volume of the unit ball $B^{2n}\subset \C^{n}$ equals:
$$\vol_n^J(B^{2n})=\frac{(4\pi)^{n/2}}{n!}\prod_{j=1}^{n}\frac{\Gamma(j+\tfrac{1}{2})}{\Gamma(j)}.$$
\end{corollary}

Notice that applying \cref{eq:Jvolball} when $n=1$ we get $\vol_1^J(B^2)=\pi=\vol_2(B^2)$, but already when $n=2$ we get $\vol^J_2(B^4)=\frac{3\pi^2}{4}\neq \frac{\pi^2}{2}=\vol_4(B^4).$  In general $ \vol^J_n$ and $\vol_{2n}$ are different, starting by the fact that the first is homogeneous of degree $n$ while the other is of degree $2n$.

\begin{proof}[Proof of \cref{eq:Jvolball}]
Let $Z=(z_1, \ldots,z_n)\in \C^n$ be a random vector filled with
independent standard complex Gaussians
$z_j=\tfrac{1}{\sqrt{2}}(\xi_{j,1}+i\xi_{j,2})$, that is,
$\xi_{j,1}, \xi_{j,2},\ldots, \xi_{n,1}, \xi_{n,2}$ are independent standard real Gaussians.
We claim that
\be\label{eq:KBall} K(Z)=\frac{1}{2\sqrt{\pi}}B^{2n}.\ee
To see this, we compute the support function of $K(Z)$. Let $u\in \C^n$. By definition, we have $h_{K(Z)}(u)=\frac{1}{2}\EE|\langle Z, u\rangle|$. Using the $U(n)$--invariance of $Z$,
we can then assume that $u=\|u\|e_1$ where $e_1$ is the first vector of the standard basis of $\C^n$. We obtain
\begin{align}h_{K(Z)}(u)
=\frac{1}{2}\EE\left|\mathrm{Re}(Z^T\overline{u})\right|
=\frac{1}{2\sqrt{2}}\|u\|\, \EE |\xi_{1,1}|
=\frac{1}{2\sqrt{2}}\|u\| \sqrt{\frac{2}{\pi}}
=\frac{1}{2\sqrt{\pi}}\|u\|
&=h_{\frac{1}{2\sqrt{\pi}} B^{2n}}(u).
\end{align}
\cref{useful_properties_for_h} gives \cref{eq:KBall}.

 Let now $M\in \C^{n\times n}$ be a random complex matrix whose columns are i.i.d.\ copies of $Z$.
 Then, \cref{thm:vitaleC} gives
\be \label{eq:absC}\EE|\det (M)|=n!\,\MV^J\big(\tfrac{1}{2\sqrt{\pi}} B^{2n}, \ldots, \tfrac{1}{2\sqrt{\pi}} B^{2n}\big)=\frac{n!}{(2\sqrt{\pi})^n}\,\vol_n^J(B^{2n}).\ee
To conclude the proof, it suffices to verify that
$\EE|\det (M)| = \prod_{j=1}^{n}\Gamma(j+\tfrac{1}{2})/\Gamma(j)$.
Note that $|\det M|=\det(W)^{\frac{1}{2}}$ and
$W=MM^*$ is a complex Wishart matrix.
Following \cite[p.~83-84]{condition}, we see that $\det(W)$ is distributed as $\tfrac{1}{2^n}\chi_{2n}^2\cdot \chi_{2n-2}^2\cdots \chi_2^2$, where each $\chi_{2j}^2$ denotes a chi--square distribution
with ${2j}$ degrees of freedom and the $\chi_{2j}^2$ are independent. Therefore,
$|\det(M)|$ has the distribution $\frac{1}{2^{n/2}}\chi_{2n}\cdot \chi_{2n-2}\cdots \chi_2.$
Recall from~\cref{eq:rhodef} that $\EE\chi_{2j}=\frac{\sqrt{2}\Gamma(j+\frac{1}{2})}{\Gamma(j)}$.
Using independence, the assertion follows.
\end{proof}

\subsection{The extension of the  $J$--volume to polytopes}\label{sec:Jext}

In this section we show that it is possible to extend the notion of $J$--volume to polytopes in $\C^n\cong \R^{2n}$.
To do so, we develop an alternative formula for the $J$--volume of zonotopes that makes sense for any polytope (\cref{thm:voltheta}). The functional we obtain is a weakly continuous, translation invariant and $U(n)$--invariant valuation, see \cref{prop:volJisval}. However, as we will see, it is not possible to continuously extend the $J$--volume from polytopes to all convex bodies (see \cref{cor:noextofJ}).
We start by introducing some terminology.

As a first step, we will give in Proposition~\ref{prop:voltheta}
an alternative way of writing the $J$--volume of zonotopes.
This involves the following quantity.

\begin{definition}\label{def:sigmaJ}
	For every $E\in G(n,2n)$, we define
	\begin{equation}
	\sigma^J(E):=\left| e_1\wedge \cdots \wedge e_n\wedge Je_1\wedge \cdots \wedge Je_n\right|\in [0,1]
	\end{equation}
	where  $\{e_1, \ldots, e_n\}$ is an orthonormal basis of $E$.
\end{definition}

One can check that this definition does not depend on the choice of an orthonormal basis.
Moreover, 
$\sigma^J$ is invariant under the action of $U(n)$ on $G(n,2n)$. 
Note that $\sigma^J(E)=1$ if and only if $E$ is \emph{Lagrangian}, i.e., if $E$ and $JE$ are orthogonal.
Moreover $\sigma^J(E)=0$ if and only if $E$ contains a complex line.

\begin{remark} Denoting by  $\theta_1(E)\leq \cdots \leq\theta_{\lfloor \frac{n}{2}\rfloor}(E)$ the \emph{K\"ahler angles} of $E$, introduced in \cite{tasaki}, one can easily verify that
\be 
\mbox{$
\sigma^J(E)=\prod_{j=1}^{\lfloor \frac{n}{2}\rfloor}(\sin \theta_i(E))^2.
$}
\ee
\end{remark}
In general, $\sigma^J(E)$ can be computed using the following lemma.

\begin{lemma}\label{lemma:Ej}
Let $z_1, \ldots, z_n\in \C^n$ be $\R$-linearly independent and denote by $E\in G(n,2n)$ its real span.
Then, writing $z_j=x_j+iy_j$  with $x_j, y_j\in \R^n$, we have
\be \label{eq:zx}|
	z_1\wc\cdots\wc z_n|=\left\|\begin{bmatrix} x_1\\y_1\end{bmatrix}\wedge\cdots \wedge \begin{bmatrix} x_n\\y_n\end{bmatrix}\right\|\cdot \sigma^J(E)^{\frac{1}{2}}.
\ee
\end{lemma}

\begin{proof}
Consider the matrices $X=(x_1, \ldots, x_n)\in\mathbb R^{n\times n}$ and
$Y=(y_1, \ldots, y_n)\in\mathbb R^{n\times n}$ with the columns $x_j,y_j$.
One can check that 
	$\det\left(\begin{smallmatrix} X&-Y\\Y&X\end{smallmatrix}\right)=|\det(X+iY)|^2$, see \cite[Lemma 5]{BLLP}.
	In particular, we can write
	\begin{align}
	|z_1\wc\cdots\wc z_n|^2=|\det(X+i Y)|^2
	&=\left|\det\begin{bmatrix} X&-Y\\Y&X\end{bmatrix}\right|\\
	&=\left|\begin{bmatrix} x_1\\y_1\end{bmatrix}\wedge\cdots \wedge\begin{bmatrix} x_n\\y_n\end{bmatrix}\wedge \begin{bmatrix} -y_1\\x_1\end{bmatrix}\wedge\cdots \wedge\begin{bmatrix} -y_n\\x_n\end{bmatrix}\right|\\
	&=\left\|\begin{bmatrix} x_1\\y_1\end{bmatrix}\wedge\cdots \wedge \begin{bmatrix} x_n\\y_n\end{bmatrix}\right\|^2\cdot \left| e_1\wedge \cdots \wedge e_n\wedge Je_1\wedge \cdots Je_n
	\right|,
	\end{align}
	where $\{e_1, \ldots, e_n\}$ is an orthonormal basis for $E$, and the last equality follows from:
	\be
	\begin{bmatrix} x_1\\y_1\end{bmatrix}\wedge\cdots \wedge\begin{bmatrix} x_n\\y_n\end{bmatrix}=\left\| \begin{bmatrix} x_1\\y_1\end{bmatrix}\wedge\cdots \wedge\begin{bmatrix} x_n\\y_n\end{bmatrix}\right\|\cdot e_1\wedge \cdots \wedge e_n.
	\ee
	The conclusion follows from \cref{def:sigmaJ}.
\end{proof}

We denote by $G(k,m)$ the Grassmannian of $k$-dimensional linear spaces in $\R^m$.
For a $k$-dimensional face $F$ of a polytope $P$ in $\R^m$, we denote by {$E_F\in G(k, m)$}
the vector space parallel to the affine span of $F$.
For $0\leq k\leq m$ we also define
\be \label{eq:zonoG}
G_k(P):=\{E\in G(k,m)\,|\, \textrm{there exists a $k$-dim.\ face $F$ of $P$ such that $E=E_F$}\}.
\ee
Let $E\in G_k(P)$.
In the case where $P=\sum_{j=1}^p \tfrac{1}{2}[-z_j, z_j]$ is a zonotope,
then all faces $F$ of $P$ such that $E_F=E$ are translates of the following ``vectorial'' face
(see McMullen~\cite{McMullen})
\begin{equation}\label{eq:facesmc}
	F_P(E):=\sum_{z_j\in E}\tfrac{1}{2}[-z_j, z_j]
\end{equation}
Where the sum runs over all $j$ such that $z_j\in E$. In this case, we have the following alternative description of $G_k(P)$: 
$$
G_k(P)=\{E\in G(k,d)\,|\,\textrm{there exist linearly independent $z_{j_1}, \ldots, z_{j_k}\in E$}\}.
$$

The next result gives an explicit expression of the $J$--volume of a zonotope.

\begin{proposition}\label{prop:voltheta}
Let $P\subset \C^n$  be a centered zonotope. Then
$$ \vol_n^J(P)=\sum_{E\in G_n(P)}\vol_n(F_P(E))\cdot\sigma^J(E)^{\frac{1}{2}}.$$
\end{proposition}
\begin{proof}
We write $P=\sum_{j=1}^p \tfrac{1}{2}[-z_j, z_j]$.
By definition, we have
$\vol_n^J(P)=\frac{1}{n!}\ell(P^{\wc n})$.
Using multilinearity and Theorem~\ref{thm:FTZC}, we can write
$$
  \Big(\sum_{j=1}^p\tfrac{1}{2}[-z_j, z_j]\Big)\wc\cdots \wc \Big(\sum_{j=1}^p\tfrac{1}{2}[-z_j, z_j] \Big) =
 \sum_{1 \le j_1,\ldots,j_n \le p}\tfrac{1}{2}[-w_{j_1,\ldots,j_n}, w_{j_1,\ldots,j_n}] ,
$$
where $w_{j_1,\ldots,j_n}:=z_{j_1}\wc \cdots \wc z_{j_n}$.
Therefore, using the linearity of the length,
$$
 \vol_n^J(P)=\frac{1}{n!}\sum_{1 \le j_1,\ldots,j_n \le p} |w_{j_1,\ldots,j_n}|
  = \sum_{j_1 <\ldots < j_n } |w_{j_1,\ldots,j_n}| .
$$
We may assume the sum runs only over the $j_1 <\ldots<j_n$ such that
the real span $E_{j_1, \ldots, j_n}$ of $z_{j_1}, \ldots, z_{j_n}$
has dimension~$n$.
We write $z_j= x_j + iy_j$ with $x_j,y_j\in\R^n$ and use
\cref{lemma:Ej} to obtain
$$
 \left|w_{j_1,\ldots,j_n}\right| =
 \left\|
\begin{bmatrix}
x_{j_1}\\y_{j_1}\end{bmatrix} \wedge\cdots \wedge \begin{bmatrix} x_{j_n}\\y_{j_n}
\end{bmatrix} \right\|\cdot \sigma^J(E_{j_1,\ldots, j_n})^{\frac{1}{2}} .
$$
Combining this and exchanging the order of summation, we arrive at
\begin{equation}\label{eq:sh}
 \vol_n^J(P)=\sum_{E\in G_n(P)}\sigma^J(E)^{\frac{1}{2}}\cdot \sum_{E_{j_1,\ldots, j_n}=E}
 \left\| \begin{bmatrix} x_{j_1}\\y_{j_1}\end{bmatrix} \wedge\cdots \wedge \begin{bmatrix} x_{j_n}\\y_{j_n}\end{bmatrix} \right\| ,
\end{equation}
where for fixed $E\in G_n(P)$, the second sum runs over all
$j_1< \ldots< j_n$ such that $E=E_{j_1, \ldots, j_n}$.
Shephard's formula \cite[Equation (57)]{Shephard}
applied to the zonoid $F_P(E)$ (see \eqref{eq:facesmc}) tells us that
\be
\sum_{E_{j_1,\ldots, j_n}=E}
 \left\| \begin{bmatrix} x_{j_1}\\y_{j_1}\end{bmatrix}\wedge\cdots\wedge
 \begin{bmatrix} x_{j_n}\\y_{j_n}\end{bmatrix} \right\|=\vol_n(F_P(E)) .
\ee
Substituting this into \cref{eq:sh} gives the statement.
\end{proof}

We now turn to a key result of this section. For this, we need to introduce more notation.
Let $F$ be a $k$-dimensional face of a polytope $P\subset\R^m$
and $N_P(F)$ denote its normal cone.
Note that $N_P(F)$ is contained in the orthogonal
complement $E_F^\perp\simeq \R^{m-k}$,
where we recall that $E_F$ denotes the vector space parallel to $F$.
We define the \emph{normal angle} of $P$ at $F$ as
\be\label{def_theta}
\Theta_P(F):=\frac{\vol_{m-k-1}(N_P(F)\cap S^{m-1})}{\vol_{m-k-1}(S^{m-k-1})}.
\ee
The $J$--volume of zonotopes can now be expressed as follows.

\begin{theorem}\label{thm:voltheta}
	Let $P\subset \C^n$  be a zonotope. Then
\be \label{eq:vnP} \vol_n^J(P)=\sum_{F\in \mathcal{F}_n(P)}\vol_n(F) \cdot \Theta_P(F)\cdot \sigma^J(E_F)^{\frac{1}{2}},\ee
where $\mathcal{F}_n(P)$ denotes the set of $n$--dimensional faces of $P$.
\end{theorem}

\begin{proof}
We will prove that the right hand side in this theorem is equal to the right hand side in \cref{prop:voltheta}.
Let $z_1, \ldots, z_p\in \C^n$ be such that $P=\sum_{j=1}^p \tfrac{1}{2}[-z_j,z_j]$ and let $E\in G_{n}(P)$.
As we discussed at the beginning of this subsection (see~\cref{eq:facesmc}), all the faces $F$ of $P$ such that $E_F=E$
are translates of the vectorial face $ F_P(E)=\sum_{z_j\in E}\tfrac{1}{2}[-z_j,z_j]$, we can thus write:
\begin{align}\label{eq:idi}
 \sum_{F\in \mathcal{F}_n(P)}\vol_n(F)\Theta_P(F) \sigma^J(E_F)^{\frac{1}{2}}&=\sum_{E\in G_{n}(P)}\vol_n(F_P(E))\sigma^J(E_F)^{\frac{1}{2}}\sum_{E_F=E}\Theta_P(F).
\end{align}
It remains to prove that for every $E\in G_n(P)$ we have $\sum_{E_F=E}\Theta_P(F)=1$.

To this end, given $E\in G_{n}(P)$, pick a nonzero $u\in E^\perp.$ Then, the set
$$P^u:=\{x\in K\,|\, h_P(u)=\langle u, x\rangle\}$$
is a face of $K$ and therefore it equals a translate of a vectorial face (see~\cref{eq:facesmc}).
More precisely, there is $v_u\in \C^n$ such that:
\begin{equation}
 P^u=v_u+\sum_{z_j\in u^\perp}\tfrac{1}{2}[-z_j,z_j],
\end{equation}
 In addition, if $F$ is a face of $P$ such that $E_F=E$,
 $F$ is a translate of $\sum_{v_j\in E}\tfrac{1}{2}[-v_j, v_j]$.
 Since $E\subset u^\perp$, the face $P^u$ contains a translate of $F$.
 Moreover, $\dim(F)=n$ and it follows that $\dim(P^u)\geq n$ which implies $\dim(N_{K}(K^u))\leq n$.
 In other words we proved that if  $E\in G_{n}(K)$ and $u\in E^\perp$, then $\dim(N_{P}(P^u))\leq n$.

 We now show that for almost all $u\in E^\perp$ we have  $\dim(N_{P}(P^u))= n$. Indeed, for this it is enough to write $E^\perp\subseteq \bigcup_{u\in E^\perp}N_K(K^u)$, thus the set $\{u\in E^\perp \mid \dim(N_P(P^u))<n\}$ is contained in a finite union of cones of dimension at most $n-1$.

Let now $S_{E^\perp}^{n-1}$ be the unit sphere in $E^\perp$. Denote by $\mathcal{F}$ the set of faces $F$ of $P$
such that $E_F\supset E$ and $\dim(N_P(F))<n$. Then, by the above reasoning,
\be
\{u\in S_{E^\perp}^{n-1}\,|\, \dim(P^u)<n\}\subseteq \bigcup_{F\in \mathcal{F}}N_P(F)\cap S_{E^\perp}^{n-1}.
\ee
Each set $N_P(F)\cap S_{E^\perp}^{n-1}$ with $F\in \mathcal{F}$ has dimension at most $n-2$.
Since the set $\mathcal{F}$ is finite, it implies, as above, that $ \{u\in S_{E^\perp}^{n-1}\,|\, \dim(P^u)<n\}$
is contained in a finite union of sets of dimension at most $n-2$, and in particular it has measure zero in $S_{E^\perp}^{n-1}$.
It follows that
$\{u\in S_{E^\perp}^{n-1}\,|\, \dim(P^u)=n\}\subset S_{E^\perp}^{n-1}$ has full measure.
Letting $u$ vary in $S_{E^\perp}^{n-1}$  the set $\{P^u\}$ exhausts all $n$--dimensional faces $F$ with $E_F=E$ and therefore:
\be \sum_{E_F=E}\Theta_P(F)=
\sum_{E_F=E}
\frac{\vol_{2n-1}(N_P(F)\cap S^{2n-1})}{\vol_{n-1}(S^{n-1})} = \frac{\vol_{n-1}(S^{n-1})}{\vol_{n-1}(S^{n-1})} = 1.\ee
This concludes the proof.
\end{proof}

We now note that the formula in \cref{thm:voltheta} still makes sense for polytopes that are not zonotopes.
We use this to define the $J$--volume on polytopes.

\begin{definition}\label{def:volJonP}
	Let $P$ be polytope in $\C^n$. We define its $J$--volume to be
	\begin{equation}
	\vol_n^J(P) := \sum_{F\in \mathcal{F}_n(P)}\vol_n(F) \cdot \Theta_P(F)\cdot \sigma^J(E_F)^{\frac{1}{2}}
	\end{equation}
	where $\mathcal{F}_n(P)$ denotes the set of $n$--dimensional faces of $P$.
\end{definition}

We next study the $J$--volume
in the framework of the theory of valuations on polytopes.
Let us first recall the notion of a valuation.
We denote by $\cP(V)$ the set of polytopes in a finite dimensional real vector space~$V$.

\begin{definition}[Valuation] \label{def:val}
A function $\nu:\cP(V)\to \R$ is called a valuation on $\cP(V)$ if
\be
 \nu(K\cup L)+\nu(K\cap L)=\nu(K)+\nu(L)
\ee
for every pair of convex polytopes $K, L\in \cP(V)$ such that $K\cup L$ is still a polytope (an analoguous definition applies for valuations on $\KK(V)$).

We call $\nu$ continuous if it is continuous with respect to the Hausdorff metric.
We say that the valuation $\nu$ is \emph{$k$--homogeneous}
if $\nu(\lambda K)=\lambda^k\nu(K)$ for every $K\in \mathcal{P}(V)$ and $\lambda>0$.
If a group~$G$ acts on $\cP(V)$, the valuation $\nu$
is said to be \emph{$G$--invariant} if $\nu(gK)=\nu(K)$ for all~$K\in \cP(V)$ and $g\in G$.
\end{definition}

Let us also define the notion of \emph{weak continuity} of a valuation on polytopes.
This corresponds to continuity of the valuation on the set of polytopes that share the same given set of normals.
More precisely,
a valuation $\nu\colon\cP(V)\to \R$ is said to be \emph{weakly continuous}
if for every finite set $U=\{u_1, \ldots, u_r\}\subset V$ of unit vectors
positively spanning $V$, i.e., $\sum_i \R_+ u_i =V$,
the function
\begin{equation}
	(\eta_1, \ldots, \eta_r)\mapsto \nu\left(\{v\in V\,|\, \langle v, u_i\rangle \leq \eta_i, \, i=1, \ldots, r\}\right)
\end{equation}
is continuous on the set $(\eta_1,\ldots, \eta_r)$ for which the argument of $\nu$ is nonempty.
One can check that a continuous valuation $\nu:\cP(V)\to \R$ is weakly continuous. The general form of weakly continuous, translation invariant valuations on $\mathcal{P}(V)$
was described in~\cite{McMWeakVal}.

We now show some properties of the $J$--volume.
\begin{proposition}\label{prop:volJisval}
The $J$--volume has the following properties.
\begin{enumerate}
\item
	The map $\vol_n^J:\cP(\C^n)\to \R$ is a weakly continuous, translation invariant valuation.
\item The valuation $\vol_{n}^J$ is $n$--homogeneous and $U(n)$--invariant.
\item Let $P\subset\R^n\subset\C^n$ be a polytope. Then $\vol_{n}^{J}(P)=\vol_{n}(P)$.
\end{enumerate}
\end{proposition}
\begin{proof}
The first item follows from \cite[Theorem~1]{McMWeakVal}.
The second item follows from the $U(n)$-invariance of $\sigma^J$
and \cref{def:volJonP}. For the third item, if $P$ is of dimension less than~$n$, both volumes are zero and there is nothing to prove.
	If $\dim(P)=n$, its only face of dimension $n$ is $P$ itself and $E_P=\R^n$.
	Moreover $\sigma^J(\R^n)=1$. Finally since $N_P(P)=(\R^n)^\perp$ we have $\Theta_P(P)=1$.
	The claim follows with \cref{def:volJonP}.
\end{proof}

The valuation $\vol_{n}^{J}$ is a special case of a so called \emph{angular valuation}, see~\cite{Wannerer}.
Let $V$ be a Euclidean space of (real) dimension $m$ and let $\varphi : G(k, m)\to \R$ be a measurable function.
The {\em associated angular valuation} $\nu_\varphi$ on $\mathcal{P}(V)$ is defined by
\begin{equation}
		\nu_\varphi(P):=\sum_{F\in \mathcal{F}_k(P)}\vol_k(F) \cdot \Theta_P(F)\cdot \varphi(E_F) ,
\end{equation}
where $\mathcal{F}_k(P)$ denotes the set of $k$--dimensional faces of a polytope $P\in\mathcal{P}(V)$.
It is known~\cite{McMWeakVal} that
$\nu_\varphi\colon \mathcal{P}(V)\to \R$
is a weakly continuous valuation.

The possibility of continuously extending an angular valuation from polytopes to convex bodies
was studied by Wannerer. The following is~\cite[Theorem 1.2]{Wannerer}.
For its statement, we recall that $G(k,m)$ can be seen as a subset of $\mathbb P(\Lambda^k V)$ via the Plücker embedding.

\begin{proposition}\label{prop:wannerer}
The angular valuation $\nu_\varphi\colon \cP(V) \to \R$ can be extended to a continuous valuation
on $\KK(V)$, if and only if $\varphi$ is the restriction to $G(k,m)$
of a homogeneous quadratic polynomial on~$\Lambda^k V$. \qed
\end{proposition}

If $n=1$, then $\sigma^J$ is the function which is constant and equal to 1. The previous proposition implies that in this case we can extend $\vol_1^J$ to a continuous valuation
on $\KK(V)$. If $n\geq 2$, however, this is not possible as we will show next.

\begin{corollary}\label{cor:noextofJ}
If $n\geq 2$, there is no continuous valuation on $\KK(\C^n)$ that is equal to $\vol_n^J$ on~$\mathcal{P}(\C^n)$.
\end{corollary}
\begin{proof}
Using the notation of \cref{prop:wannerer}, we have $\vol_n^J=\nu_{(\sigma^J)^{1/2}}$.
We identify $\C^n\cong\R^{2n}$ and let $J$ be the standard complex structure on it.
Consider the homogeneous quadratic polynomial
$
 p\colon \Lambda^n\R^{2n}\to \Lambda^{2n}\R^{2n}, \, w \mapsto w\wedge J w .
$
From \cref{def:sigmaJ}, $\sigma^J(w) = \vert p(w)\vert $ for
$w\in G(n,2n)$ (in the Pl\"ucker embedding).
Suppose there were a homogeneous quadratic polynomial
$q\colon \Lambda^n\R^{2n}\to \R$ such that we have
$\vert p(w)\vert^\frac12 =q(w)$ for all $w\in G(n,2n)$.
Let us show that this leads to a contradiction,
which will complete the proof by~\cref{prop:wannerer}. First of all, we notice that $q(w)$ must be a nonnegative polynomial and that we have $\vert p(w)\vert  = q(w)^2$ on $G(n,2n)$.

Let $e_1,\ldots,e_n\in\mathbb C^n$ be the standard basis, so that $\vert e_1\wedge \cdots \wedge e_n\wedge Je_1\wedge \cdots \wedge Je_n\vert = 1$. We define the curve $w(\theta):=(\cos(\theta) e_1 + \sin(\theta) Je_2)\wedge e_2\wedge \cdots \wedge e_n$ in $G(n,2n)$ for $\theta \in [0,\pi]$.
This curve interpolates between a Lagrangian plane (for $\theta = 0$) and a plane, which contains a complex line (for $\theta = \pi$). We have that
\begin{align}
p(w(\theta)) &= (\cos(\theta) e_1 + \sin(\theta) Je_2)\wedge e_2\wedge \cdots \wedge e_n\wedge (\cos(\theta) Je_1 - \sin(\theta) e_2)\wedge e_2\wedge \cdots \wedge e_n\\
& = \cos(\theta)^2 (e_1\wedge\cdots \wedge e_n\wedge Je_1\wedge \cdots \wedge J e_n),
\end{align}
and so $\vert p(w(\theta))\vert =\cos(\theta)^2$. If we have $\vert p(w(\theta))\vert = q(w(\theta))^2$, then $q(w(\theta))= \cos (\theta)$, because $q$ is nonnegative.  Since $q$ is a quadratic form and by the definition of $w(\theta)$, there are $a,b,c\in\mathbb R$ such that $q(w(\theta)) = a\cos(\theta)^2 + b\cos(\theta)\sin(\theta) + c\sin(\theta)^2$  for all $\theta$. Thus, we have an equality of functions
$ a\cos(\theta)^2 + b\cos(\theta)\sin(\theta) + c\sin(\theta)^2=\cos(\theta)$. It can be checked that such an equality is not possible, so our assumption was wrong and $(\sigma^J)^{1/2}$ cannot be the restriction of the square of a quadratic form to $G(n,2n)$. \cref{prop:wannerer} implies the assertion.
\end{proof}

From the proof, we see that if we remove the square root in \cref{def:volJonP},
we could extend the valuation continuously to convex bodies.
This leads to the notion of \emph{Kazarnovskii's pseudovolume}~\cite{Ka1}.
We use the expression found in~\cite{alesker} in the proof of his Proposition~3.3.1.
The normalization constant can be determined using the fact that it agrees with the classical volume on
$\R^n\subset\C^n$ just like the $J$--volume.

\begin{definition}\label{def:Kaza}
The Kazarnovskii's pseudovolume $\vol_n^K$ is given for any polytope $P\subset \C^n$ by the formula
\begin{equation}
\vol_n^K(P)=\sum_{F\in \mathcal{F}_n(P)}\vol_n(F) \cdot \Theta_K(F)\cdot \frac{1}{(\omega_n)^2}\vol_{2n}(B(E_F)+JB(E_F)),
\end{equation}
where $\mathcal{F}_n(P)$ denotes the set of $n$-dimensional faces of $P$, and as before $B(E_F)$ denotes the unit ball of $E_F$, and $\omega_n:=\vol_n(B(\R^n))$.
\end{definition}

In our setting we prove the following, to be compared to \cref{def:volJonP}.

\begin{proposition}
	For any polytope $P\in\mathcal{P}(\C^n)$ the Kazarnovskii pseudovolume is given by
	\begin{equation}
	\vol_n^K(P)=\sum_{F\in \mathcal{F}_n(P)}\vol_n(F) \cdot \Theta_K(F)\cdot \sigma^J(E_F),
	\end{equation}
	where $\mathcal{F}_n(P)$ denotes the set of $n$-dimensional faces of $P$;
\end{proposition}

\begin{proof}
We need to prove that for any $E\in G(n,2n)$ we have
\begin{equation}\label{eq:volBsigma}
\vol_{2n}(B(E)+JB(E))=(\omega_n)^2 \sigma^J(E).
\end{equation}
Using \cref{thm:FTZCdet} we write
\begin{align}
\vol_{2n}(B(E)+JB(E))	&=\frac{1}{(2n)!}\ell\left(\big(B(E)+JB(E)\big)^{\wedge 2n} \right)	\\
&=\frac{1}{(2n)!}\sum_{j=0}^{2n}\binom{2n}{j} \ell\left((B(E))^{\wedge j}\wedge (JB(E))^{\wedge(2n-j)}\right)
\end{align}
where we used from \cref{th:lengthLF} that $\ell$ is linear. Since $\dim(B(E))=\dim(JB(E))=n$,
we see from \cref{wedge_is_zero} that $(B(E))^{\wedge j}=0$ whenever $j> n$
and that $(JB(E))^{\wedge(2n-j)}=0$ whenever $j< n$. In other words,
only the index $j=n$ contributes to the sum
and we get
\begin{equation}\label{eq:thisisalabel}
\vol_{2n}(B(E)+JB(E))=\frac{1}{(n!)^2}\, \ell\left((B(E))^{\wedge n}\wedge (JB(E))^{\wedge n}\right).
\end{equation}
Next let $X\in E$ be a random vector such that $K(X)=B(E)$ (for instance, \cref{eq:ball}
shows that we can take $X$ to be $\sqrt{2\pi}$ times a standard Gaussian vector in $E$) and
let $X_1,\ldots,X_{n}$ be i.i.d.\ copies of $X$. 
Let $e_1,\ldots,e_n$ be an orthonormal basis of $E$.
Note that we have
\begin{equation}
X_1\wedge\cdots\wedge X_n=\pm \Vert X_1\wedge\cdots\wedge X_n \Vert\  e_1\wedge \cdots \wedge e_n.
\end{equation}
With this in mind, \cref{eq:thisisalabel} gives
\begin{align}
\vol_{2n}(B(E)+JB(E))	&=\frac{1}{(n!)^2}\; \EE |X_1\wedge\cdots\wedge X_n\wedge J X_{n+1}\wedge \cdots \wedge X_{2n}|	\\
	&=\frac{1}{(n!)^2}\left( \EE\Vert X_1\wedge\cdots\wedge X_n \Vert \right)^2\sigma^J(E)
\end{align}
Again, by \cref{th:IVL}, we have $\EE\Vert X_1\wedge\cdots\wedge X_n \Vert = n! \vol_n(B(E))=n!\omega_n$
and this gives~\cref{eq:volBsigma}, which concludes the proof.
\end{proof}

\bibliographystyle{alpha}
\bibliography{literature}

\end{document}